\documentclass[11pt]{article}%
\usepackage{amsmath, upgreek}
\usepackage{amssymb}
\usepackage{amscd}
\usepackage{xspace}
\usepackage{verbatim}
\newcommand{\be}{\begin{equation}}
\newcommand{\bel}[1]{\begin{equation}\label{#1}}
\newcommand{\ee}{\end{equation}}
\newcommand{\barr}{\begin{eqnarray}}
\newcommand{\earr}{\end{eqnarray}}
\newcommand{\bars}{\begin{eqnarray*}}
\newcommand{\ears}{\end{eqnarray*}}

\newtheorem{subn}{\name}

\newcommand{\bsn}[1]{\def\name{#1}\begin{subn}}
\newcommand{\esn}{\end{subn}}
\newtheorem{sub}{\name}[section]
\newcommand{\bs}{\begin{sub}}
\newcommand{\es}{\end{sub}}
\newcommand{\bsl}[1]{\begin{sub}\label{#1}}
\newcommand{\bth}[1]{\def\name{Theorem}
\begin{sub}\label{t:#1}}
\newcommand{\blemma}[1]{\def\name{Lemma}
\begin{sub}\label{l:#1}}
\newcommand{\bcor}[1]{\def\name{Corollary}
\begin{sub}\label{c:#1}}
\newcommand{\bdef}[1]{\def\name{Definition}
\begin{sub}\label{d:#1}}
\newcommand{\bprop}[1]{\def\name{Proposition}
\begin{sub}\label{p:#1}}

\newcommand{\BA}{\begin{array}}
\newcommand{\EA}{\end{array}}
\newcommand{\BAN}{\renewcommand{\arraystretch}{1.2}
\setlength{\arraycolsep}{2pt}\begin{array}}
\newcommand{\BAV}[2]{\renewcommand{\arraystretch}{#1}
\setlength{\arraycolsep}{#2}\begin{array}}
\newcommand{\BSA}{\begin{subarray}}
\newcommand{\ESA}{\end{subarray}}
\newcommand{\BAL}{\begin{aligned}}
\newcommand{\EAL}{\end{aligned}}
\newcommand{\BALG}{\begin{alignat}}
\newcommand{\EALG}{\end{alignat}}
\newcommand{\BALGN}{\begin{alignat*}}
\newcommand{\EALGN}{\end{alignat*}}

\newcommand{\abs}[1]{\left |#1\right |}

\def\angb<#1>{\langle #1 \rangle}

\newcommand{\opname}[1]{\mbox{\rm #1}\,}
\newcommand{\supp}{\opname{supp}}
\newcommand{\dist}{\opname{dist}}
\newcommand{\myfrac}[2]{{\displaystyle \frac{#1}{#2} }}
\newcommand{\myint}[2]{{\displaystyle \int_{#1}^{#2}}}

\newcommand{\prt}{\partial}

\def\ga{\alpha}     \def\gb{\beta}

            \def\gl{\lambda}
\def\gm{\mu}

      \def\gw{\omega}
                
     \def\Gd{\Delta}

\def\Gw{\Omega}              

      \def\CF{{\mathcal F}}

   \def\BBN {\mathbb N}    
   \def\BBR {\mathbb R}

\providecommand{\U}[1]{\protect\rule{.1in}{.1in}}
\newtheorem{theorem}{Theorem}[section]

\newtheorem{definition}[theorem]{Definition}

\newtheorem{remark}[theorem]{Remark}

\newenvironment{proof}[1][Proof]{\textbf{#1.} }{\hfill\rule{0.5em}{0.5em}}
{\catcode`\@=11\global\let\AddToReset=\@addtoreset
\AddToReset{equation}{section}

\AddToReset{theorem}{section}

\begin{document}
\title{Quasilinear and Hessian type equations with\\  exponential reaction and measure data \footnote{To appear {\bf Arch. Rat. Mech. Anal.}}}
\author{
 {\bf Nguyen Quoc Hung\thanks{ E-mail address: Hung.Nguyen-Quoc@lmpt.univ-tours.fr}}\\[1mm]
 {\bf Laurent V\'eron\thanks{ E-mail address: Laurent.Veron@lmpt.univ-tours.fr, Corresponding author.}}\\[2mm]
{\small Laboratoire de Math\'ematiques et Physique Th\'eorique, }\\
{\small  Universit\'e Fran\c{c}ois Rabelais,  Tours,  FRANCE}}
\date{}  
\maketitle 
\tableofcontents
\abstract{ We prove existence results concerning equations of the type $-\Gd_pu=P(u)+\gm$ for $p>1$ and  $F_k[-u]=P(u)+\gm$ with $1\leq k<\frac{N}{2}$ in a bounded domain $\Omega$ or the whole $\mathbb{R}^N$, where $\gm$ is a positive Radon measure and $P(u)\sim e^{au^\beta}$ with $a>0$ and $\beta\geq 1$. Sufficient conditions for existence are expressed in terms of the  fractional maximal potential of $\gm$. Two-sided estimates on the solutions are obtained in terms of some precise Wolff potentials of $\gm$. Necessary conditions are obtained in terms of Orlicz capacities. We also establish existence results for a general Wolff potential equation under the form $u={\bf W}_{\alpha,p}^R[P(u)]+f$ in $\mathbb{R}^N$, where $0<R\leq \infty$ and $f$ is a positive integrable function. } \\
 
 \smallskip

\vspace{1mm}
\hspace{.05in}
 \section{Introduction} 
 
 Let $\Omega\subset\BBR^N$ be either a bounded domain or the whole $\mathbb{R}^N$, $p>1$ and $k\in\{1,2,...,N\}$. We denote by 
 $$\Gd_pu:=div\left(\abs{\nabla u}^{p-2} \nabla u\right)$$ 
 the p-Laplace operator and by 
 $$ F_k[u]=\sum_{1\leq j_1<j_2<...<j_k\leq N}\lambda_{ j_1}\lambda_{ j_2}...\lambda_{ j_k}$$
 the k-Hessian operator where $\gl_1,...,\gl_N$ are the eigenvalues of the Hessian matrix $D^2u$.
Let  $\mu$ be a positive Radon measure in $\Gw$; our aim is to study the existence of nonnegative solutions to the following boundary value problems if $\Gw$ is bounded, 
 \bel{EQ1}\BA {ll}
 -\Gd_pu=P(u)+\gm\qquad&\text{in }\Gw,\\
 \phantom{ -\Gd_p}
 u=0\qquad&\text{on }\prt\Gw,
\EA \ee
and 
 \bel{EQ2}\BA {ll}
 F_k[-u]=P(u)+\gm\qquad&\text{in }\Gw,\\
 \phantom{  F_k[-]}
 u= \varphi\qquad&\text{on }\prt\Gw,
\EA \ee
 where $P$ is an exponential function. If $\Gw=\BBR^N$, we consider the same equations, but the boundary conditions are replaced by $\inf_{\BBR^N}u=0$. When $P(r)=r^q$ with $q>p-1$,  Phuc and Verbitsky published a seminal article \cite {PhVe} on the solvability of the corresponding problem (\ref{EQ1}). They obtained necessary and sufficient conditions  involving Bessel capacities or  Wolff potentials. 
For example, assuming that $\Omega$ is bounded, they proved that if $\mu$ has compact support in $\Omega$ it is equivalent to solve  
(\ref{EQ1}) with $P(r)= r^q$, or to have
\bel{EQ3}\BA {ll}
\mu(E)\leq c\text{Cap}_{\mathbf{G}_p,\frac{q}{q+1-p}}(E)\qquad\text{for all compact set }E\subset\Omega,
\EA \ee
where $c$ is a suitable positive constant and $\text{Cap}_{\mathbf{G}_p,\frac{q}{q+1-p}}$ a Bessel capacity, or to have
\bel{EQ4}\BA {ll}
\myint{B}{}\left({\bf W}^{2R}_{1,p}[\mu_B](x)\right)^{q}dx\leq C\mu (B)\qquad\text{for all ball }B \text{ s.t. }
B\cap{\rm supp}\mu\neq\emptyset,
\EA \ee
where $R={\rm diam}(\Omega)$. Other conditions are expressed in terms of Riesz potentials and maximal fractional potentials. Their construction is based upon sharp estimates of solutions of the non-homogeneous 
problem
\bel{EQ5}\BA {ll}
 -\Gd_pu=\gw\qquad&\text{in }\Gw,\\
 \phantom{ -\Gd_p}
 u=0\qquad&\text{on }\prt\Gw,
\EA \ee
 for positive measures $\gw$. We refer to \cite{Bi1,Bi2,Bi3, Bi4,BiDe,Gr, SeZo} for the previous studies of these and other related results. Concerning the k-Hessian operator in a bounded $(k-1)$-convex domain $\Omega$,  they proved that if $\mu$ has compact support and $||\varphi||_{L^\infty(\partial\Omega)}$ is small enough, the corresponding  problem \eqref{EQ2} with $P(r)=r^q$ with $q>k$ admits a nonnegative solution if and only if 
\bel{EQ6}\BA {ll}
\mu(E)\leq c\text{Cap}_{\mathbf{G}_{2k},\frac{q}{q-k}}(E)\qquad\text{for all compact set }E\subset\Omega,
\EA \ee
or equivalently
\bel{EQ7}\BA {ll}
\myint{B}{}\left[{\bf W}^{2R}_{\frac{2k}{k+1},k+1}[\mu_B(x)]\right]^{q}dx\leq C\mu (B)\qquad\text{for all ball }B \text{ s.t. }
B\cap{\rm supp}\mu\neq\emptyset.
\EA \ee
The results concerning the  linear case $p=2$ and $k=1$, can be found in \cite{AdPi,BaPi,V}.
The main tools in their proofs are derived from recent advances in potential theory for nonlinear elliptic equations obtained
by Kilpelainen and Mal\'y \cite{KiMa1,KiMa2}, Trudinger and Wang  \cite{TW1,TW2,TW3}, and Labutin \cite{La} thanks to whom the authors first provide global pointwise estimates for solutions of the homogeneous Dirichlet problems in terms of Wolff potentials of suitable order.   \medskip

For $s>1$, $0\leq\ga<\frac{N}{s}$, $\eta\geq 0$ and $0<T\leq\infty$, we recall that the {\it $T$-truncated Wolff potential} of a positive Radon measure $\gm$ is defined in $\BBR^N$ by
\begin{align}\label{EQ8}
{\bf W}^T_{\ga,s}[\mu](x)=\int_{0}^{T}\left(\frac{\mu (B_t(x))}{t^{N-\ga s}}\right)^{\frac{1}{s-1}}\frac{dt}{t},
\end{align}
the {\it $T$-truncated Riesz potential} of a positive Radon measure $\gm$  by
\begin{align}\label{EQ8a}{\bf I}^T_{\ga}[\mu](x)=\myint{0}{T}\frac{\mu (B_t(x))}{t^{N-\ga }}\frac{dt}{t},
\end{align}
and the {\it $T$-truncated $\eta$-fractional maximal potential} of $\gm$ by
\begin{align}
\label{EQ9}{\bf M}^\eta_{\ga,T}[\mu](x)=\sup\left\{\frac{\mu (B_t(x))}{t^{N-\ga }h_\eta(t)}:0<t\leq T\right\},
\end{align}
where $h_\eta(t)=(-\ln t)^{-\eta}\chi_{(0,2^{-1}]}(t)+(\ln 2)^{-\eta}\chi_{[2^{-1},\infty)}(t)$. If $\eta=0$, then $h_\eta=1$ and we denote by ${\bf M}_{\ga,T}[\mu]$ the corresponding {\it $T$-truncated fractional maximal potential} of $\gm$. We also denote  by ${\bf W}_{\alpha,s}[\mu]$ (resp ${\bf I}_{\alpha}[\mu]$, ${\bf M}_{\ga}^\eta[\mu]$ )  the {\it $\infty$-truncated Wolff potential (resp Riesz Potential, $\eta-$ fractional maximal potential) } of $\gm$. 
When the measures are only defined in an open subset $\Omega\subset\BBR^N$, they are naturally extended by $0$ in $\Gw^c$.  For $l\in\BBN^*$,  we define the {\it $l$-truncated exponential function}
\bel{EQ10}\BA {ll}\displaystyle
H_l(r)=e^r-\sum_{j=0}^{l-1}\frac{r^j}{j!},
\EA \ee
and for $a>0$ and $\beta\geq 1$, we set
\bel{EQ11}\BA {ll}
P_{l,a,\beta}(r)=H_l(ar^\gb).
\EA \ee
We put
\bel{EQ13}
Q_p(s) = \left\{ \begin{array}{ll}
    \sum\limits_{q = l}^\infty  {\frac{{{s^{\frac{\beta q}{p-1}}}}}{{{q^{\frac{{\beta q}}{{p - 1}}}q!}}}} \quad&\text{if }\,p \ne 2, \\ [5mm]
    H_l(s^\beta)&\text{if }\,p = 2, 
 \end{array} \right.   
 \ee
$Q_p^*(r)=\max \left\{rs-Q_p(s):s\geq 0\right\}$ is the complementary function to $Q_p$, and
define the corresponding Bessel and Riesz capacities respectively by
\bel{EQ14}
\text{Cap}_{{\mathbf{G}_{\alpha p } },Q^*_p}(E) = \inf \left\{ {\int_{\mathbb{R}^N} {Q^*_p(f)}dx :{\mathbf{G}_{\alpha p} }\ast f \geq {\chi _E},f \geq 0, Q^*_p(f) \in {L^1}({\mathbb{R}^N})} \right\},
 \ee
 and 
 \bel{EQ15}
 \text{Cap}_{{\mathbf{I}_{\alpha p} },Q^*_p}(E) = \inf \left\{ {\int_{\mathbb{R}^N} {Q^*_p(f)}dx :{\mathbf{I}_{\alpha p} }\ast f \geq {\chi _E},f \geq 0, Q^*_p(f) \in {L^1}({\mathbb{R}^N})} \right\},
  \ee
where $\mathbf{G}_{\alpha p }(x)=\CF^{-1}\left((1+\abs.^2)^{-\frac{\alpha p}{2}}\right)(x)$ is the Bessel kernel of order $p$ and $I_{\alpha p}(x)=(N-\alpha p)^{-1}|x|^{-(N-\alpha p )}$. \medskip

The expressions $a \wedge b$ and $a \vee b$ stand for $\min\{a,b\}$ and $\max\{a,b\}$ respectively. We denote by $B_r$ the ball of center $0$ and radius $r>0$.  Our main results are the following theorems.
\begin{theorem}\label{MT1}
Let $1<p<N$, $a>0$, $l\in \mathbb{N}^*$ and $\beta\geq 1$ such that $l\beta>p-1$. Let $\Omega\subset \mathbb{R}^N$ be a bounded domain. If $\mu$ is a nonnegative Radon measure in $\Omega$, there exists $M>0$ depending on $N,p,l,a,\beta$ and $diam\,(\Omega)$ (the diameter of $\Omega$) such that if 
\begin{equation*}
||{\bf M}_{ p,2\,diam\,(\Omega)}^{\frac{{(p - 1)(\beta  - 1)}}{\beta }}[\mu ]|{|_{{L^\infty }({\mathbb{R}^N})}} \le M,
\end{equation*}
and  $\omega=M||{\bf M}_{ p,2\,diam (\Omega)}^{\frac{(p-1)(\beta-1)}{\beta}} [1] ||_{{L^\infty }(\mathbb{R}^N)}^{-1}+\mu$ with $c_p={1 \vee {4^{\frac{{2 - p}}{{p - 1}}}}} $, then $ P_{l,a,\beta}\left(2c_pK_1{\bf W}_{1 ,p}^{2\,diam\,(\Omega)}[\omega ]\right)$  is integrable in $\Gw$ and  the following  Dirichlet problem
  \begin{equation}\label{MT1a}
  \left. \begin{array}{ll}
    - \Delta _p u = P_{l,a,\beta}(u)+\mu \qquad&\text{in }\;\Omega,  \\ 
    \phantom{    - \Delta _p}
   u = 0 \qquad&\text{on }\;\prt\Omega,
   \end{array} \right.
  \end{equation}
   admits a nonnegative renormalized solution $u$ which satisfies  
   \begin{equation}\label{MT1b}
   {u(x)} \le 2c_pK_1{\bf W}_{1 ,p}^{2\,diam\,(\Omega)}[\omega ](x) \quad\forall x\in \Omega.
   \end{equation}
 The role of  $K_1=K_1(N,p)$ will be made explicit in Theorem \ref{TH4}.
 
   Conversely, if \eqref{MT1a} admits a nonnegative renormalized solution $u$ and $P_{l,a,\beta}(u)$ is integrable in $\Gw$, then for any compact set $K\subset \Omega$, there exists a positive constant $C$ depending on $N,p,l,a, \beta$ and $\dist(K,\partial\Omega)$ such that 
   \begin{equation}\label{MT1c}
   \int_E {P_{l,a,\gb}(u)dx}  + \mu (E) \le C \text{Cap}_{{\mathbf{G}_{ p}},Q_p^*}(E)~~\textrm{ for all Borel sets } E\subset K.
   \end{equation}
   Furthermore, $u\in W^{1,p_1}_0(\Omega)$ for all $1\le p_1<p$. 
 \end{theorem}

When $\Gw=\BBR^N$, we have a similar result provided $\mu $ has compact suppport.
\begin{theorem}\label{MT2}
Let $1<p<N$, $a>0$, $l\in \mathbb{N}^*$ and $\beta\geq 1$ such that $l\beta>\frac{N(p-1)}{N-p}$ and $R>0$. If $\mu$ is a nonnegative Radon measure in $\mathbb{R}^N$ with $\supp (\mu)\subset B_{R}$ there exists $M>0$ depending on $N,p,l,a,\beta$ and $R$ such that if 
\begin{equation*}
||{\bf M}_{ p}^{\frac{{(p - 1)(\beta  - 1)}}{\beta }}[\mu ]|{|_{{L^\infty }({\mathbb{R}^N})}} \le M,
\end{equation*}
and $\omega=M||{\bf M}_{p}^{\frac{(p-1)(\beta-1)}{\beta}}[\chi_{B_R}]||_{L^\infty(\mathbb{R}^N)}^{-1}\chi_{B_R}+\mu$, then $P_{l,a,\beta}\left(2c_pK_1{\bf W}_{1 ,p}[\omega ]\right)$ is integrable in $\BBR^N$ and the following problem 
  \begin{equation}\label{MT2a}
  \left. \begin{array}{ll}
   \phantom{,} - \Delta _p u = P_{l,a,\beta}(u)+\mu ~\text{in }\;\mathcal{D}'(\mathbb{R}^N),\\
    \phantom{}
     \inf_{\mathbb{R}^N}u=0,
   \end{array} \right.
  \end{equation}
   admits a p-superharmonic solution $u$ which satisfies 
   \begin{equation}\label{MT2b}
   {u(x)} \le 2c_pK_1{\bf W}_{1 ,p}[\omega ](x) \quad\forall x\in \mathbb{R}^N,
   \end{equation}
($c_p$ and $K_1$ as in Theorem \ref{MT1}).

   Conversely, if \eqref{MT2a} has a solution $u$ and $P_{l,a,\beta}(u)$ is locally integrable in  $\mathbb{R}^N$, then there exists a positive constant $C$ depending on $N,p,l,a, \beta$  such that 
   \begin{equation}\label{MT2c}
   \int_E {P_{l,a,\gb}(u)dx}  + \mu (E) \le C \text{Cap}_{{\mathbf{I}_{ p}},Q_p^*}(E)~~\forall E\subset \mathbb{R}^N, E ~\textrm{Borel}.  \end{equation}
   Furthermore, $u\in W^{1,p_1}_{loc}(\mathbb{R}^N)$ for all $1\le p_1<p$. 
 \end{theorem}

Concerning the $k$-Hessian operator we recall some notions introduced by Trudinger and Wang \cite{TW1,TW2,TW3}, and we follow their notations. For $k=1,...,N$ and $u\in C^2(\Omega)$ the k-Hessian operator $F_k$ is defined by 
 $$F_k[u]=S_k(\lambda(D^2u)),$$
 where $\lambda(D^2u)=\lambda=(\lambda_1,\lambda_2,...,\lambda_N)$ denotes the eigenvalues of the Hessian matrix of second partial derivatives $D^2u$ and $S_k$ is the k-th elementary symmetric polynomial that is \[{S_k}(\lambda ) = \sum\limits_{1 \le {i_1} < ... < {i_k} \le N} {{\lambda _{{i_1}}}...{\lambda _{{i_k}}}}. \]
It is straightforward that \[{F_k}[u] = {\left[ {{D^2}u} \right]_k},\]
 where in general $[A]_k$ denotes the sum of the k-th principal minors of a matrix $A=(a_{ij})$.
  In order that there exists a smooth k-admissible function which vanishes on $\partial \Omega$, the boundary $\partial \Omega$ must satisfy a uniformly (k-1)-convex condition, that is 
  \begin{equation*}
  S_{k-1}(\kappa ) \geq c_0 >0~on ~~ \partial\Omega,
  \end{equation*}
for some positive constant $c_0$, where $\kappa= (\kappa_1,\kappa_2,...,\kappa_{n-1})$ denote the principal curvatures of $\partial \Omega$ with respect to its inner normal. We also denote by $\Phi^k(\Omega)$ the class of upper-semicontinuous functions $\Gw\to[Õ-\infty,\infty)$ which are $k$-convex, or subharmonic in the Perron sense (see Definition \ref{k-conv}).\
In this paper we prove the following theorem (in which expression $\mathbb E[q]$ is the largest integer less or equal to $q$)


\begin{theorem}\label{MT3}Let $k\in \{1,2,...,\mathbb E[N/2]\}$ such that $2k<N$, $l\in \mathbb N^*$, $\beta\geq 1$ such that $l\beta>k$ and $a>0$. Let $\Omega$ be a bounded uniformly (k-1)-convex domain in $\mathbb{R}^N$. Let $\varphi$ be a nonnegative continuous function on $\partial \Omega$ and  $\mu=\mu_1+f$ be a nonnegative Radon measure where $\mu_1$ has compact support in $\Omega$ and $f\in L^q(\Omega)$ for some $q>\frac{N}{2k}$. Let $K_2=K_2(N,k)$ be the constant $K_2$ which appears in Theorem \ref{TH5}. Then, there exist positive constants $b$, $M_1$ and $M_2$ depending on $N,k,l,a,\beta$ and $diam\,(\Omega)$ such that, if $\max_{\partial \Omega}\varphi \le M_2$ and
\begin{equation*}
||{\bf M}_{ 2k,2diam\,(\Omega)}^{\frac{{k(\beta  - 1)}}{\beta }}[\mu ]|{|_{{L^\infty }({\mathbb{R}^N})}} \le M_1,
\end{equation*}
then $P_{l,a,\beta}\left(2K_2{\bf W}_{\frac{2k}{k+1},k+1}^{2\,diam\,(\Omega)}[\mu ] + b\right)$ is integrable in $\Omega$ and the following  Dirichlet problem 
    \begin{equation}\label{MT3a}
    \left. \begin{array}{ll}
      F_k[-u] = P_{l,a,\beta}(u)+\mu \qquad&\text {in }\;\Omega,  \\ 
      \phantom{F_k[-]}
     u = \varphi&text {on }\;\prt\Omega,  \\ 
     \end{array} \right.
    \end{equation}
     admits a nonnegative  solution $u$, continuous near $\partial \Omega$, with $-u\in \Phi^k(\Omega)$ which satisfies      \begin{equation}\label{MT3b}
     {u(x)} \le 2 K_2{\bf W}_{\frac{2k}{k+1},k+1}^{2diam\,(\Omega)}[\mu ](x) +b\qquad\forall x\in \Omega.
     \end{equation}
     
     Conversely, if \eqref{MT3a} admits a nonnegative solution $u$, continuous near $\partial \Omega$, such that $-u\in \Phi^k(\Omega)$ and $P_{l,a,\beta}(u)$ is integrable in $\Gw$, then for any compact set $K\subset \Omega$, there exists a positive constant $C$ depending on $N,k, l,a,\beta$ and $dist(K,\partial\Omega)$ such that there holds 
     \begin{equation}\label{MT3c}
        \int_E {P_{l,a,\beta}(u)dx}  + \mu (E) \le C \text{Cap}_{{\mathbf{G}_{2k}},Q_{k+1}^*}(E)\qquad\forall E\subset K, E\text { Borel}, 
 \end{equation}
        where $Q_{k+1}(s)$ is defined by (\ref{EQ13}) with $p=k+1$, $Q_{k+1}^*$ is its complementary function 
and $\text{Cap}_{{\mathbf{G}_{2k}},Q_{k+1}^*}(E)$ is defined accordingly by (\ref{EQ14}).
 \end{theorem}
 
 The following extension holds when $\Gw=\BBR^N$.

\begin{theorem}\label{MT4}Let $k\in \{1,2,...,\mathbb E[N/2]\}$ such that $2k<N$, $l\in \mathbb N^*$, $\beta\geq 1$ such that $l\beta>\frac{Nk}{N-2k}$ and $a>0$, $R>0$. If $\mu$ is a nonnegative Radon measure in $\mathbb{R}^N$ with $\supp (\mu)\subset B_{R}$ there exists $M>0$ depending on $N,k,l,a,\beta$ and $R$ such that if 
\begin{equation*}
||{\bf M}_{ 2k}^{\frac{{k(\beta  - 1)}}{\beta }}[\mu ]|{|_{{L^\infty }({\mathbb{R}^N})}} \le M,
\end{equation*}
and $\omega=M||{\bf M}_{2k}^{\frac{k(\beta-1)}{\beta}}[\chi_{B_R}]||_{L^\infty(\mathbb{R}^N)}^{-1} \chi_{B_R}+\mu$, then $P_{l,a,\beta}\left(2K_2{\bf W}_{\frac{2k}{k+1},k+1}[\omega ] \right)$ is integrable in $\mathbb{R}^N$ ($K_2$  as in Theorem \ref{MT3}) and the following  Dirichlet problem 
    \begin{equation}\label{MT4a}
    \left. \begin{array}{ll}
          F_k[-u] = P_{l,a,\beta}(u)+\mu ~~\text {in }\;\mathbb{R}^N,  \\ 

     \inf_{\mathbb{R}^N}u=0,  \\ 
     \end{array} \right.
    \end{equation}
     admits a nonnegative  solution $u$ with $-u\in \Phi^k(\mathbb{R}^N)$ which satisfies
     \begin{equation}\label{MT4b}
     {u(x)} \le 2 K_2{\bf W}_{\frac{2k}{k+1},k+1}[\omega ](x)\qquad\forall x\in \mathbb{R}^N.
     \end{equation}

Conversely, if \eqref{MT4a} admits a nonnegative solution $u$ with  $-u\in \Phi^k(\mathbb{R}^N)$ and $P_{l,a,\beta}(u)$ locally integrable in $\mathbb{R}^N$, then there exists a positive constant $C$ depending on $N,k, l,a,\beta$  such that there holds 
     \begin{equation}\label{MT4c}
        \int_E {P_{l,a,\beta}(u)dx}  + \mu (E) \le C \text{Cap}_{{\mathbf{I}_{2k}},Q_{k+1}^*}(E)\qquad\forall E\subset \mathbb{R}^N, E\text { Borel}. 
 \end{equation}
        where $\text{Cap}_{{\mathbf{I}_{2k}},Q_{k+1}^*}(E)$ is defined accordingly by (\ref{EQ15}).
 \end{theorem}

The four previous theorems are connected  to the following results which deals with a class of nonlinear {\it Wolff integral equations}.

  \begin{theorem}\label{MT5}
    Let $\alpha>0$, $p>1$, $a>0$, $\varepsilon>0$, $R>0$, $l\in \mathbb{N}^*$ and $\beta\geq 1$ such that  $l\beta>p-1$ and $0<\alpha p <N$. Let $f$ be a nonnegative measurable in $\mathbb{R}^N$ with the property that $\mu_1=P_{l,a+\varepsilon,\beta}(f)$ is locally integrable in $\mathbb{R}^N$ and  $\mu\in \mathfrak{M}^+(\mathbb{R}^N)$.
There exists $M>0$ depending on $N,\alpha,p,l,a,\beta,\varepsilon$ and $R$ such that if 
\begin{equation}\label{14061}||{\bf M}_{\alpha p,R}^{\frac{{(p - 1)(\beta  - 1)}}{\beta }}[\mu ]|{|_{{L^\infty }({\mathbb{R}^N})}} \le M~~\text{ and }~~ ||{\bf M}_{\alpha p,R}^{\frac{{(p - 1)(\beta  - 1)}}{\beta }}[\mu_1 ]|{|_{{L^\infty }({\mathbb{R}^N})}} \le M,
\end{equation}
     then there exists a nonnegative function $u$ such that   $P_{l,a,\beta}(u)$ is locally integrable in $\mathbb{R}^N$ which satisfies
  \begin{equation}\label{MT5a}
   u = {\bf W}_{\alpha ,p}^{R}[P_{l,a,\beta}(u)+\mu] + f ~\textrm{ in }~\mathbb{R}^N,
  \end{equation} 
  and
  \begin{equation}\label{MT5b}{u} \le F:=2c_p {\bf W}_{\alpha ,p}^{R}[\omega_1 ] + 2c_p{\bf W}_{\alpha,p}^{R}[\omega_2]+ f,~~~ {P_{l,a,\beta}\left(F\right) \in L^1_{loc}(\mathbb{R}^N)},
      \end{equation}
  where $\omega_1=M||{\bf M}_{\alpha p,R}^{\frac{(p-1)(\beta-1)}{\beta}} [1] ||^{-1}_{{L^\infty }(\mathbb{R}^N)}+\mu$ and $\omega_2=M||{\bf M}_{\alpha p,R}^{\frac{(p-1)(\beta-1)}{\beta}} [1] ||^{-1}_{{L^\infty }(\mathbb{R}^N)}+\mu_1$.
  
   Conversely, if \eqref{MT5a} admits a nonnegative solution $u$ and $P_{l,a,\beta}(u)$ is locally integrable in $\mathbb{R}^N$, then there exists a positive constant $C$ depending on $N,\alpha,p,l,a,\beta$ and $R$ such that there holds
   \begin{equation}\label{MT5c}
   \int_E P_{l,a,\beta}(u)dx+ \int_EP_{l,a+\varepsilon,\beta}(f)dx+\mu(E)  \le C\text{Cap}_{{\mathbf{G}_{\alpha p}},Q_p^*}(E)\quad\forall E\subset \mathbb{R}^N,\, E\text{ Borel}.
   \end{equation}
      \end{theorem}
      
      When $R=\infty$ in the above theorem, we have a similar result provided $f$ and $\gm$ have compact support in $\BBR^N$.
     \begin{theorem}\label{MT6}  Let $\alpha>0$, $p>1$, $a>0$, $\varepsilon>0$, $R>0$, $l\in \mathbb{N}^*$ and $\beta\geq 1$ such that $0<\alpha p <N$ and $l\beta >\frac{N(p-1)}{N-\alpha p}$. There exists $M>0$ depending on $N,\alpha,p,l,a,\beta,\varepsilon$ and $R$ such that if $f$ is a nonnegative measurable function in $\mathbb{R}^N$ with support in  $B_{R}$ such that  $\mu_1=P_{l,a+\varepsilon,\beta}(f)$ is locally integrable in $\mathbb{R}^N$ and  $\mu$ is a positive measure in $\mathbb{R}^N$ with support in $B_{R}$ which verify
     \begin{equation}
     \label{14062}||{\bf M}_{\alpha p}^{\frac{{(p - 1)(\beta  - 1)}}{\beta }}[\mu ]|{|_{{L^\infty }({\mathbb{R}^N})}} \le M~~\text{ and }~~ ||{\bf M}_{\alpha p}^{\frac{{(p - 1)(\beta  - 1)}}{\beta }}[\mu_1 ]|{|_{{L^\infty }({\mathbb{R}^N})}} \le M,
     \end{equation}
        then there exists a nonnegative function $u$ such that   $P_{l,a,\beta}(u)$ is integrable in $\mathbb{R}^N$ which satisfies
     \begin{equation}\label{MT5d}
      u = {\bf W}_{\alpha ,p}[P_{l,a,\beta}(u)+\mu] + f ~\textrm{ in }~\mathbb{R}^N,
     \end{equation} 
     and
     \begin{equation}\label{MT5e}{u} \le F:=2c_p {\bf W}_{\alpha ,p}[\omega_1 ] + 2c_p{\bf W}_{\alpha,p}[\omega_2]+ f,~~~ {P_{l,a,\beta}\left(F\right) \in L^1(\mathbb{R}^N)},
         \end{equation}
     where $\omega_1=M||{\bf M}_{\alpha p}^{\frac{(p-1)(\beta-1)}{\beta}} [\chi_{B_{R}}] ||^{-1}_{{L^\infty }(\mathbb{R}^N)}\chi_{B_{R}}+\mu$ and $\omega_2=M||{\bf M}_{\alpha p}^{\frac{(p-1)(\beta-1)}{\beta}} [\chi_{B_{R}}] ||^{-1}_{{L^\infty }(\mathbb{R}^N)}\chi_{B_{R}}+\mu_1.$

      Conversely, if \eqref{MT5d} admits a nonnegative solution $u$ such that $P_{l,a,\beta}(u)$ is integrable in $\mathbb{R}^N$, then there exists a positive constant $C$ depending on $N,\alpha,p,l,a,\beta$  such that there holds
      \begin{equation}
      \label{MT5f}\int_E P_{l,a,\beta}(u)dx+\int_E P_{l,a,\beta}(f)dx+\mu(E)  \le C\text{Cap}_{{\mathbf{I}_{\alpha p}},Q_p^*}(E)\quad\forall E\subset \mathbb{R}^N,\, E\text{ Borel}.
      \end{equation}
   \end{theorem}
   
As an application of the Wolff integral equation we can notice that $\alpha=1$, equation (\ref{MT5d}) is equivalent to 
 $$-\Gd_p(u-f)=P_{l,a,\beta}(u)+\mu\qquad\text{in }\BBR^N.
 $$
When $\alpha=\frac{2k}{k+1}$ and $p=k+1$, it is equivalent to
  $$F_k[-u+f]=P_{l,a,\beta}(u)+\mu\qquad\text{in }\BBR^N.
 $$
 If $p=2$ equation (\ref{MT5d}) becomes linear. If we set $\gamma=2\ga$, then 
 $$\BA{ll}
 {\bf W}_{\alpha ,2}[\omega](x)=\myint{0}{\infty}\omega(B_t(x))\myfrac{dt}{t^{N-\gamma+1}}\\[4mm]\phantom{{\bf W}_{\gamma ,2}[\omega](x)}=\myint{\BBR^N}{}\left(\myint{\abs{x-y}}{\infty}\myfrac{dt}{t^{N-\gamma+1}}\right)d\gm(y)
 \\[4mm]\phantom{{\bf W}_{\gamma ,2}[\omega](x)}
 =\frac{1}{N-\gamma}\myint{\BBR^N}{}\myfrac{d\omega(y)}{|x-y|^{N-\gamma}}\\[4mm]\phantom{{\bf W}_{\gamma ,2}[\omega](x)}=I_{\gamma}\ast\omega,
 \EA$$
where $I_\gamma$ is the Riesz kernel of order $\gamma$. Thus (\ref{MT5d}) is equivalent to 
  $$(-\Gd)^{\alpha}(u-f)=P_{l,a,\beta}(u)+\mu\qquad\text{in }\BBR^N.
 $$
 \begin{remark}
 In case $\Omega$ is a bounded open set, uniformly bounded of sequence $\{u_n\}$ \eqref{1232} is essential  for the existence of solutions of equations \eqref{MT1a}, \eqref{MT3a} and \eqref{MT5a}. Moreover, conditions $l\beta>p-1$ in Theorem \ref{MT1},  \ref{MT5} and  $l\beta>k$ in Theorem \ref{MT3} is necessary so as to get \eqref{1232} from iteration schemes  \eqref{ite-1}. Besides, in case $\Omega=\mathbb{R}^N$, equation  \eqref{MT2a} in Theorem \ref{MT2} $($ \eqref{MT4a} in Theorem \ref{MT4}, \eqref{MT5d} in Theorem \ref{MT6} resp.$)$ has nontrivial solution on $\mathbb{R}^N$ if and only if  $l\beta>\frac{N(p-1)}{N-p}$ $($ $l\beta>\frac{Nk}{N-2k}$,  $l\beta>\frac{N(p-1)}{N-\alpha p}$ resp.$)$. In fact, here we only need to consider equation \eqref{MT2a}. Assume that $l\beta\leq \frac{N(p-1)}{N-p}$, using Holder inequality we have
 $P_{l,a,\beta}(u)\geq c u^\gamma$ where $p-1<\gamma\leq \frac{N(p-1)}{N-p}$, so we get from Theorem \eqref{TH4}.
 \begin{equation*}
 u\geq K {\bf W}_{1,p}[cu^\gamma+\mu] ~~\text{ in } \mathbb{R}^N
 \end{equation*}
 for some constant $K$. 
 Therefore, we can verify that 
\begin{equation*}
\int_E u^\gamma dx+\mu(E)  \le C \text{Cap}_{{\mathbf{I}_{p}},\frac{\gamma}{\gamma-p+1}}(E)\qquad\forall E\subset \mathbb{R}^N,\, E\text{ Borel}.
\end{equation*}  
see Theorem \ref{TH1}, where $C$ is a constant and $\text{Cap}_{\mathbf{I}_p,\frac{\gamma}{\gamma-p+1}}$ is  a Riesz capacity. \\
Since $N\leq \frac{p \gamma}{\gamma-p+1}$ $ (\Leftrightarrow p-1<\gamma\leq \frac{N(p-1)}{N-p})$, $\text{Cap}_{\mathbf{I}_p,\frac{\gamma}{\gamma-p+1}}(E)=0$ for all Borel set $E$, see \cite{AH}. \\
Immediately, we deduce  $u \equiv 0$ and $\mu\equiv 0$. 
 \end{remark}
  
\section{Estimates on potentials and Wolff integral equations} 
We denote by $B_r(a)$ the ball of center $a$ and radius $r>0$, $B_r=B_r(0)$ and by $\chi_{E}$ the characteristic function of a set $E$. The next estimates are crucial in the sequel.

\begin{theorem}Let $\alpha>0$, $p>1$ such that $0<\alpha p<N$. \\\label{VHVth1}\textbf{1.} 
There exists a positive constant $c_1$, depending  only  on $N,\alpha,p$ such that  for all $\mu\in \mathfrak{M}^+(\mathbb{R}^N)$ and  $q\geq p-1$, $0<R\le \infty$ we have
\begin{equation}\label{ine0}
(c_1q)^{ - \frac{q}{p - 1}}\int_{\mathbb{R}^N} \left( {\bf I}_{\alpha p}^R[\mu ](x) \right)^{\frac{q}{p - 1}}dx
 \leq \int_{\mathbb{R}^N} \left( {\bf W}_{\alpha ,p}^R[\mu ](x) \right)^{q}dx 
\leq (c_1q)^q\int_{\mathbb{R}^N} \left( {\bf I}^R_{\alpha p}[\mu ](x) \right)^{\frac{q}{p - 1}}dx,
\end{equation}

\noindent\textbf{2.} Let $R>0$. There exists a positive constant $c_2$, depending  only  on $N,\alpha,p, R$ such that  for all $\mu\in \mathfrak{M}^+(\mathbb{R}^N)$ and  $q\geq p-1$ we have
\begin{equation}\label{ine2}
(c_2q)^{ - \frac{q}{p - 1}}\int_{\mathbb{R}^N} \left( {\bf G}_{\alpha p}[\mu ](x) \right)^{\frac{q}{p - 1}}dx
 \leq \int_{\mathbb{R}^N} \left( {\bf W}_{\alpha ,p}^R[\mu ](x) \right)^{q}dx 
\leq (c_2q)^q\int_{\mathbb{R}^N} \left( {\bf G}_{\alpha p}[\mu ](x) \right)^{\frac{q}{p - 1}}dx,
\end{equation}
where ${\bf G}_{\alpha p}[\mu ]:=\mathbf{G}_{\alpha p}\ast\mu$ denotes the Bessel potential of order $\alpha p$ of $\mu$.\smallskip

\noindent\textbf{3.} There exists a positive constant $c_3$, depending  only  on $N,\alpha, R$ such that  for all $\mu\in \mathfrak{M}^+(\mathbb{R}^N)$ and  $q\geq 1$ we have
\begin{equation}\label{ine3}
c_3^{ - q}\int_{\mathbb{R}^N} \left({\bf G}_{\alpha }[\mu ](x) \right)^qdx  \leq \int_{\mathbb{R}^N} \left( {\bf I}_{\alpha}^R[\mu ](x) \right)^qdx  \leq c_3^q\int_{\mathbb{R}^N} \left( {\bf G}_{\alpha }[\mu ](x) \right)^qdx.
\end{equation}
\end{theorem}
\begin{proof} Note that  $W^R_{\frac{\alpha}{2},2}[\mu]=I^R_{\alpha}[\mu]$.
We can find proof of \eqref{ine3} in  \cite[Step 3, Theorem 2.3]{VHV}. By \cite[Step 2, Theorem 2.3]{VHV}, there is $c_4>0$ such that \begin{equation}\label{17051}
\int_{{\mathbb{R}^N}} {{{\left( {{\bf W}_{\alpha ,p}^R[\mu ](x)} \right)}^q}dx}  \geq c_4^q\int_{{\mathbb{R}^N}} {{{\left( {{{\bf M}_{\alpha p,R}}[\mu ](x)} \right)}^{\frac{q}{{p - 1}}}}dx} ~~\forall q \geq  p-1,\; 0<R\le \infty \text{ and }\; \mu\in \mathfrak{M}^+(\mathbb{R}^N).
\end{equation}
We recall that ${\bf M}_{\alpha p,R}[\mu ]={\bf M}^0_{\alpha p,R}[\mu ]$ by (\ref{EQ9}). Next we show that for all 
$q \geq  p-1$, $0<R \le \infty$ and $\mu\in \mathfrak{M}^+(\mathbb{R}^N)$ there holds
\begin{equation}\label{17052}
\int_{{\mathbb{R}^N}} {{{\left( {{{\bf M}_{\alpha p,R}}[\mu ](x)} \right)}^{\frac{q}{{p - 1}}}}dx} \; \geq {\left( {{c_5}q} \right)^{ - q}}\int_{{\mathbb{R}^N}} {{{\left( {{\bf W}_{\alpha ,p}^R[\mu ](x)} \right)}^q}dx},
\end{equation}
for some positive constant $c_5$ depending on $N,\ga,p$. Indeed, we denote $\mu_n$ by $\chi_{B_n}\mu$ for $n\in\mathbb{N}^*$. By \cite[Theorem 1.2]{HoJa} or \cite[Proposition 2.2]{VHV}, there exist  constants $c_6=c_6(N,\alpha,p)>0$, $a=a(\alpha,p)>0$ and  $\varepsilon_0=\varepsilon(N,\alpha,p)$ such that for all $n\in \mathbb{N^*}$, $t>0$, $0<R\le \infty$ and $0<\varepsilon<\varepsilon_0$, there holds
 \[\left| {\left\{ {{\bf W}_{\alpha ,p}^R{\mu _n} > 3t } \right\}} \right| \le {c_6}\exp \left( { - a{\varepsilon ^{ - 1}}} \right)\left| {\left\{ {{\bf W}_{\alpha ,p}^R{\mu _n} > t } \right\}} \right| + \left| {\left\{ {{{\left( {{{\bf M}_{\alpha p,R}}{\mu _n}} \right)}^{\frac{1}{{p - 1}}}} > \varepsilon t } \right\}} \right|.\]
Multiplying by $qt^{q-1}$ and integrating over $(0,\infty)$, we obtain 
\[\BA {ll}\displaystyle
\int_0^\infty  {q{t^{q - 1}}\left| {\left\{ {{\bf W}_{\alpha ,p}^R{\mu _n} > 3t} \right\}} \right|dt}  \le {c_6}\exp \left( { - a{\varepsilon ^{ - 1}}} \right)\int_0^\infty  {q{t^{q - 1}}\left| {\left\{ {{\bf W}_{\alpha ,p}^R{\mu _n} > t} \right\}} \right|dt}  \\[4mm]\phantom{---------------------}\displaystyle
+ \int_0^\infty  {q{t^{q - 1}}\left| {\left\{ {{{\left( {{{\bf M}_{\alpha p,R}}{\mu _n}} \right)}^{\frac{1}{{p - 1}}}} > \varepsilon t} \right\}} \right|dt},
\EA\]
which implies 
\begin{equation*}
{\varepsilon ^q}\left( {{3^{ - q}} - {c_6}\exp \left( { - a{\varepsilon ^{ - 1}}} \right)} \right)\int_{\mathbb{R}^N}{{{\left( {{\bf{W}}_{\alpha ,p}^R[{\mu _n}](x)} \right)}^q}dx}  \le \int_{{\mathbb{R}^N}} {{{\left( {{{\bf{M}}_{\alpha p,R}}{\mu _n}} \right)}^{\frac{q}{{p - 1}}}}dx}.
\end{equation*} 
We see that $\mathop {\sup }\limits_{0<\varepsilon<\varepsilon_0 }{\varepsilon ^q}\left( {{3^{ - q}} - {c_6}\exp \left( { - a{\varepsilon ^{ - 1}}} \right)} \right)\geq  (c_7 q)^{-q}$ for some constant $c_7$ which does not depend on $q$. Therefore \eqref{17052} follows by Fatou's lemma. 
Hence, it is easy to obtain \eqref{ine0} from \eqref{17051} and \eqref{17052}. At end, we obtain \eqref{ine2} from \eqref{ine0} and \eqref{ine3}. 
\end{proof}\medskip\medskip

The next result is proved in \cite{VHV}.
\begin{theorem} \label{VHVth2}
Let $\alpha>0$, $p>1$, $0\le\eta<p-1$, $0<\alpha p<N$ and $L>0$.
 Set $\delta  = \frac{1}{2}{\left( {\frac{{p - 1 - \eta }}{{12(p - 1)}}} \right)^{\frac{{p - 1}}{{p - 1 - \eta }}}}\alpha p\log (2)$. 
Then there exists $C(L)>0$, depending on 
 $N$, $\alpha$, $p$, $\eta$  and $L$ such that for any $R\in (0,\infty]$, $\mu\in\mathfrak{M}^+(\mathbb{R}^N)$,  any  $a\in\mathbb{R}^N$ and $0<r\leq L$, there holds
\begin{equation}\label{ine1}
\frac{1}{{|{B_{2r}}(a)|}}\int_{{B_{2r}}(a)} {\exp \left( {\delta \frac{{{{\left( {{\bf{W}}_{\alpha ,p}^R[{\mu _{{B_r}(a)}}](x)} \right)}^{\frac{{p - 1}}{{p - 1 - \eta }}}}}}{{||{\bf{M}}_{\alpha p,R}^\eta [{\mu _{{B_r}(a)}}]||_{{L^\infty }({B_r}(a))}^{\frac{1}{{p - 1 - \eta }}}}}} \right)dx}  \le C(L),
\end{equation}
where  $\mu_{B_{r}(a)}=\chi _{B_{r}(a)}\mu$. Furthermore, if $\eta=0$, $C$ is independent of $L$.
\end{theorem}


\begin{theorem}\label{th2}Let $\alpha>0$, $p>1$ with $0<\alpha p<N$, $\beta\geq 1$ and $R>0$. Assume $\mu\in \mathfrak{M}^+(\mathbb{R}^N)$ satisfies 
 \begin{equation}\label{05061}
 ||{\bf M}_{\alpha p,R}^{\frac{(p-1)(\beta-1)}{\beta}} [\mu] |{|_{{L^\infty }(\mathbb{R}^N)}} \le 1,
 \end{equation}
and set $\omega= ||{\bf M}_{\alpha p,R}^{\frac{(p-1)(\beta-1)}{\beta}} [1] ||_{{L^\infty }(\mathbb{R}^N)}^{-1}+\mu $. Then there exist positive constants $C$, $\delta_0$ and $c$ independent on $\mu$ such that $\exp \left( {\delta_0{{\left( { {\bf W}_{\alpha ,p}^R\left[ \omega  \right]} \right)}^{\beta}}} \right)$ is locally integrable in $\mathbb{R}^N$,
\begin{equation}\label{23051}
{\left\| { {\bf W}_{\alpha ,p}^R\left[ {\exp \left( {\delta_0{{\left( { {\bf W}_{\alpha ,p}^R\left[ \omega  \right]} \right)}^{\beta}}} \right)} \right]} \right\|_{{L^\infty }(\mathbb{R}^N)}} \le C,
\end{equation}
and
\begin{equation}\label{05062}
{ {\bf W}_{\alpha ,p}^R\left[ {\exp \left( {\delta_0{{\left( { {\bf W}_{\alpha ,p}^R\left[ \omega  \right]} \right)}^{\beta}}} \right)} \right]}\le c  {\bf W}_{\alpha ,p}^R[\omega] ~~\text{ in } \mathbb{R}^N.
\end{equation}
  
\end{theorem}
\begin{proof} Let $\delta$ be as in Theorem \ref{VHVth2}. From \eqref{05061}, we have 
 \begin{equation*}
 ||{\bf M}_{\alpha p,R}^{\frac{(p-1)(\beta-1)}{\beta}} [\omega] |{|_{{L^\infty }(\mathbb{R}^N)}} \le 2.
 \end{equation*}
 Let $x\in \mathbb{R}^N$. Since $\omega ({B_t}(y)) \le 2{t^{N - \alpha p}}{h_{\frac{(p-1)(\beta -1)}{\beta}} }(t)$, for all $r\in (0,R)$ and $y\in \mathbb{R}^N$ we have 
 \begin{align*}
 \nonumber
   {\bf W}_{\alpha ,p}^R\left[ \omega  \right](y) &=  {\bf W}_{\alpha ,p}^r\left[ \omega \right](y) + \int_r^R {{{\left( {\frac{{\omega ({B_t}(y))}}{{{t^{N - \alpha p}}}}} \right)}^{\frac{1}{{p - 1}}}}\frac{{dt}}{t}}  \\ 
   &\le\nonumber  {\bf W}_{\alpha ,p}^r\left[ \omega  \right](y) + 2^{\frac{1}{p-1}}\int_{r \wedge 2^{-1}}^{2^{-1}} { {{{( - \ln t)}^{ - \frac{\beta-1}{\beta} }}} \frac{{dt}}{t}}  +2^{\frac{1}{p-1}} \int_{2^{-1}}^{R \vee 2^{-1}} {{( - \ln t)}^{ - \frac{\beta-1}{\beta} }\frac{{dt}}{t}}  \\ 
   &\le  {\bf W}_{\alpha ,p}^r\left[ \omega \right](y) + c_8{( - \ln (r \wedge 2^{-1}))^{\frac{1}{\beta}}} +c_8.
 \end{align*}
Thus, 
\begin{equation}\label{25031}{\left( { {\bf W}_{\alpha ,p}^R\left[ \omega \right](y)} \right)^{\beta}} \le 3^{\beta-1}{\left( { {\bf W}_{\alpha ,p}^r\left[ \omega \right](y)} \right)^{\beta}}+c_9\ln \left( {\frac{1}{{r \wedge 2^{-1}}}} \right)+c_9.
\end{equation}
Let $\theta\in (0,2^{-\frac{\beta}{p-1}}]$, since $\exp{(\frac{a+b}{2})}\le \exp{(a)}+\exp{(b)}$ for all $a, b\in \mathbb{R}$,  we get from \eqref{25031} 
\begin{align}
\nonumber
\exp \left( \theta\delta 3^{-\beta}{  {{\left( { {\bf W}_{\alpha ,p}^R\left[ \omega \right](y)} \right)}^{\beta}}} \right) &\leq \exp \left( { \delta 2^{-\frac{\beta}{p-1}} {{\left( { {\bf W}_{\alpha ,p}^r\left[\omega  \right](y)} \right)}^{\beta}}} \right) 
+ c_{10}\exp \left( {\theta c_{11}\ln \left( {\frac{1}{{r \wedge 2^{-1}}}} \right)} \right)\\&= \exp \left( { \delta 2^{-\frac{\beta}{p-1}} {{\left( { {\bf W}_{\alpha ,p}^r\left[ \omega  \right](y)} \right)}^{\beta}}} \right) 
+ c_{10}\left({r \wedge 2^{-1}}\right)^{-\theta c_{11}}.\label{1231}
\end{align}
For $r>0$, $0<t\le r$, $y\in B_{r}(x)$ there holds $B_t(y)\subset B_{2r}(x)$. Thus, ${\bf W}_{\alpha ,p}^r[\omega]={\bf W}_{\alpha ,p}^r[\omega_{B_{2r}(x)}]$ in $B_r(x)$. Then, using \eqref{ine1} in Theorem \ref{VHVth2} with $\eta=\frac{(p-1)(\beta-1)}{\beta}$ and $L=2R$ we get
\[\begin{array}{ll}\displaystyle\int_{{B_r}(x)} {\exp \left( { \delta 2^{-\frac{\beta}{p-1}} {{\left( { {\bf W}_{\alpha ,p}^r\left[ \omega  \right]} \right)}^{\gamma}}} \right)}  = \int_{{B_r}(x)} {\exp \left( { \delta 2^{-\frac{\beta}{p-1}} {{\left( { {\bf W}_{\alpha ,p}^r\left[ {{\omega _{{B_{2r}}(x)}}} \right]} \right)}^{\gamma}}} \right)} 
 \le c_{12}{r^N}.\end{array}\]
Therefore, taking $\theta  = 2^{-\frac{\beta}{p-1}} \wedge \frac{{\alpha p}}{{2c_{11}}}$, we deduce from \eqref{1231}
\begin{align*}
{\bf W}_{\alpha ,p}^R\left[ {\exp \left( {\theta \delta 3^{-\beta}{{\left( { {\bf W}_{\alpha ,p}^R\left[ \omega  \right]} \right)}^{\gamma}}} \right)} \right](x)&\le \int_0^R {\left( {c_{12}{r^{\alpha p}}}+ {c_{13}{{\left( {r \wedge 2^{-1}} \right)}^{ - \theta c_{11}}}{r^{\alpha p}}} \right)^{\frac{1}{{p - 1}}}}\frac{{dr}}{r}
\\&\le \int_0^R {\left( {c_{12}{r^{\alpha p}}}+ {c_{13}{{\left( {r \wedge 2^{-1}} \right)}^{ - \frac{\alpha p}{2}}}{r^{\alpha p}}} \right)^{\frac{1}{{p - 1}}}}\frac{{dr}}{r}
\\&\le   c_{14}.
\end{align*}
Hence, we get \eqref{23051} with $\delta_0=\left(2^{-\frac{\beta}{p-1}} \wedge \frac{{\alpha p}}{{2c_{11}}}\right)\delta 3^{-\beta}$; we also get \eqref{05062} since ${\bf W}_{\alpha,p}^R[\omega]\geq c_{15}$ for some positive constant $c_{15}>0$.
\end{proof}\medskip


We recall that $H_l$ and $P_{l,a,\beta}$ have been defined in (\ref{EQ10}) and (\ref{EQ11}).
\begin{theorem}
\label{th3}Let $\alpha>0$, $p>1$, $l\in \mathbb{N}^*$ and $\beta\geq 1$ such that $0<\alpha p <N$,  $l\beta>\frac{N(p-1)}{N-\alpha p}$ and  $R>0$. Assume that $\mu\in \mathfrak{M}^+(\mathbb{R}^N)$ has support in $B_R$ and verifies 
\begin{equation}
\label{08061}||{\bf M}_{\alpha p}^{\frac{(p-1)(\beta-1)}{\beta}} [\mu] |{|_{{L^\infty }(\mathbb{R}^N)}} \le 1,
\end{equation}
and set $\omega=||{\bf M}_{\alpha p}^{\frac{(p-1)(\beta-1)}{\beta}} [\chi_{B_R}] ||^{-1}_{{L^\infty }(\mathbb{R}^N)}\chi_{B_R}+\mu$. Then there exist $C=C(N,\alpha,p,l,\beta,R)>0$ and $\delta_1=\delta_1(N,\alpha,p,l,\beta,R)>0$ such that $H_l\left(\delta_1\left({\bf W}_{\alpha,p}[\omega]\right)^{\beta}\right)$ is integrable in $\mathbb{R}^N$ and
\begin{equation}\label{04063}
{\bf W}_{\alpha,p}\left[H_l\left(\delta_1\left({\bf W}_{\alpha,p}[\omega]\right)^{\beta}\right)\right](x)\le C{\bf W}_{\alpha,p}[\omega](x) ~~\forall ~x \in \mathbb{R}^N.
\end{equation}
\end{theorem}
\begin{proof} We have from \eqref{08061} 
\begin{equation}\label{04061}
||{\bf M}_{\alpha p}^{\frac{(p-1)(\beta-1)}{\beta}} [\omega] |{|_{{L^\infty }(\mathbb{R}^N)}}\le 2.
\end{equation} 
In particular, $\omega(B_R)\le c_{16}$. Let $\delta_1>0$ and $x\in \mathbb{R}^N$ fixed.  We split the Wolff potential ${\bf W}_{\alpha,p}[\omega]$ into lower and upper parts defined by 
\begin{equation*}
{\bf L}_{\alpha,p}^{t}[\omega](x)=\int_{t}^{+\infty}\left(\frac{\omega(B_r(x))}{r^{N-\alpha p}}\right)^{\frac{1}{p-1}}\frac{dr}{r},
\end{equation*} 
and \begin{equation*}
{\bf W}_{\alpha,p}^t[\omega](x)=\int_{0}^{t}\left(\frac{\omega(B_r(x))}{r^{N-\alpha p}}\right)^{\frac{1}{p-1}}\frac{dr}{r}. 
\end{equation*} 
Using the convexity we have $$H_l\left(\delta_1\left({\bf W}_{\alpha,p}[\omega]\right)^{\beta}\right)\le H_l\left(\delta_12^{\beta}\left({\bf L}^t_{\alpha,p}[\omega]\right)^{\beta}\right)+H_l\left(\delta_12^{\beta}\left({\bf W}^t_{\alpha,p}[\omega]\right)^{\beta}\right).$$
Thus, 
\begin{equation*}
{\bf W}_{\alpha,p}\left[H_l\left(\delta_1\left({\bf W}_{\alpha,p}[\omega]\right)^{\beta}\right)\right](x) \le c_{17} \int_{0}^{+\infty}\left(\frac{\omega_t^1(B_t(x))}{t^{N-\alpha p}}\right)^{\frac{1}{p-1}}\frac{dt}{t}+c_{17} \int_{0}^{+\infty}\left(\frac{\omega_t^2(B_t(x))}{t^{N-\alpha p}}\right)^{\frac{1}{p-1}}\frac{dt}{t},
\end{equation*}
where $d\omega_t^1=H_l\left(\delta_12^{\beta}\left({\bf L}_{\alpha,p}^{t}[\omega]\right)^{\beta}\right)dx$ and $d\omega_t^2=H_l\left(\delta_12^{\beta}\left({\bf W}_{\alpha,p}^t[\omega]\right)^{\beta}\right)dx$. 
 Inequality \eqref{04063} will follows from the two inequalities below,
\begin{equation}\label{31051}
\int_{0}^{+\infty}\left(\frac{\omega_t^1(B_t(x))}{t^{N-\alpha p}}\right)^{\frac{1}{p-1}}\frac{dt}{t}\le c_{18} {\bf W}_{\alpha,p}[\omega](x),
\end{equation}
and 
\begin{equation}\label{31052}
\omega^2_t(B_t(x))\le c_{18} \omega(B_{4t}(x)).
\end{equation}
{\it Step 1: Proof of (\ref{31051})}.  Since $B_r(y)\subset B_{2r}(x)$ for $y\in B_t(x)$ and $r\geq t$, there holds 
\begin{equation*}
{\bf L}_{\alpha,p}^t[\omega](y)\le \int_{t}^{+\infty}\left(\frac{\omega(B_{2r}(x))}{r^{N-\alpha p}}\right)^{\frac{1}{p-1}}\frac{dr}{r}=2^{\frac{N-\alpha p}{p-1}}{\bf L}_{\alpha,p}^{2t}[\omega](x).
\end{equation*}
It follows 
\begin{equation*}\label{31053}
\omega^1_t(B_t(x))\le |B_1(0)|t^NH_l\left(\delta_1c_{19}\left({\bf L}_{\alpha,p}^{2t}[\omega](x)\right)^{\beta}\right).
\end{equation*}
Thus,
\begin{equation}\label{08062}
\int_{0}^{+\infty}\left(\frac{\omega_t^1(B_t(x))}{t^{N-\alpha p}}\right)^{\frac{1}{p-1}}\frac{dt}{t}\le c_{20}\int_{0}^{\infty} A_t(x)dt,
\end{equation} 
where 
$$A_t(x)=\left(t^{\alpha p}H_l\left(\delta_1c_{19}\left({\bf L}_{\alpha,p}^{2t}[\omega](x)\right)^{\beta}\right)\right)^{\frac{1}{p-1}}\frac{1}{t}.$$  
Since $H_l(s)\le s^l \exp(s)$ for all $s\geq 0$,
\begin{align*}
A_t(x)&\le c_{21}\left(t^{\alpha p}\left({\bf L}_{\alpha,p}^{2t}[\omega](x)\right)^{l\beta}\exp\left(\delta_1c_{19}\left({\bf L}^{2t}_{\alpha,p}[\omega](x)\right)^{\beta}\right)\right)^{\frac{1}{p-1}}\frac{1}{t}
\\&=  c_{21}t^{\frac{\alpha p}{p-1}-1}\left({\bf L}_{\alpha,p}^{2t}[\omega](x)\right)^{\frac{l\beta-p+1}{p-1}}\exp\left(\delta_1c_{22}\left({\bf L}_{\alpha,p}^{2t}[\omega](x)\right)^{\beta}\right){\bf L}_{\alpha,p}^{2t}[\omega](x).
\end{align*} 
Now we estimate ${\bf L}^{2t}_{\alpha,p}[\omega]$. \smallskip

\noindent{\it Case 1}: $t\in (0,1)$. From \eqref{04061} we deduce
\begin{align*}
{\bf L}_{\alpha,p}^{2t}[\omega](x)&\le\int_{t/2}^{1/2}\left(\frac{\omega(B_s(x))}{s^{N-\alpha p}}\right)^{\frac{1}{p-1}}\frac{ds}{s}+\int_{1/2}^{\infty}\left(\frac{\omega(B_s(x))}{s^{N-\alpha p}}\right)^{\frac{1}{p-1}}\frac{ds}{s}
\\&\le c_{23}\int_{t/2}^{1/2}(-ln(s))^{-1+\frac{1}{\beta}}\frac{ds}{s}+\int_{1/2}^{\infty}\left(\frac{\omega(B_R)}{s^{N-\alpha p}}\right)^{\frac{1}{p-1}}\frac{ds}{s}
\\&\le c_{24}\left(-\ln (t/2)\right)^{\frac{1}{\beta}},
\end{align*} 
which implies 
\begin{align*}
A_t(x)&\le  c_{25}t^{\frac{\alpha p}{p-1}-1}\left(-\ln (t/2)\right)^{\frac{l\beta-p+1}{\beta(p-1)}}\exp\left(\delta_1c_{26}(-\ln(t/2))\right){\bf L}_{\alpha,p}^{2t}[\omega](x)\\&=
 c_{27}t^{\frac{\alpha p}{p-1}-1}\left(-\ln (t/2)\right)^{\frac{l\beta-p+1}{\beta(p-1)}}t^{-\delta_1c_{26}}{\bf L}_{\alpha,p}^{2t}[\omega](x).
\end{align*} 
We take $\delta_1\le \frac{1}{2c_{26}}\left(\frac{\alpha p}{p-1}-1\right)$ and obtain
\begin{equation}
A_t(x)\le c_{28}{\bf L}_{\alpha,p}^{2t}[\omega](x)~~\forall t\in (0,1).
\end{equation}
{\it Case 2}: $t\geq 1$. We have 
\begin{eqnarray*}
{\bf L}_{\alpha,p}^{2t}[\omega](x)\le \int_{2t}^{\infty}\left(\frac{\omega(B_R)}{s^{N-\alpha p}}\right)^{\frac{1}{p-1}}\frac{ds}{s}= c_{29}t^{-\frac{N-\alpha p}{p-1}},
\end{eqnarray*}
thus 
\begin{align*}
A_t(x)&\le c_{30}t^{\frac{\alpha p}{p-1}-1}t^{-\frac{(l\beta-p+1)(N-\alpha p)}{(p-1)^2}}\exp\left(\delta_1c_{31}t^{-\frac{\beta(N-\alpha p)}{p-1}}\right){\bf L}_{\alpha,p}^{2t}[\omega](x)
\\&\le   c_{32}t^{-1-\gamma}{\bf L}_{\alpha,p}^{2t}[\omega](x),
\end{align*}
where $\gamma=\frac{1}{p-1}\left(\frac{l\beta(N-\alpha p)}{p-1}-N\right)>0$.\\
Therefore, $A_t(x)\le c_{33} (t\vee 1)^{-1-\gamma}{\bf L}_{\alpha,p}^{2t}[\omega](x)$ for all $t>0$.
Therefore, from \eqref{08062} \begin{equation*}
\int_{0}^{+\infty}\left(\frac{\omega_t^1(B_t(x))}{t^{N-\alpha p}}\right)^{\frac{1}{p-1}}\frac{dt}{t}\le c_{34}\int_{0}^{\infty} (t\vee 1)^{-1-\gamma}{\bf L}_{\alpha,p}^{2t}[\omega](x)dt.
\end{equation*} 
Using Fubini Theorem we get 
\begin{align*}
\int_{0}^{+\infty}\left(\frac{\omega_t^1(B_t(x))}{t^{N-\alpha p}}\right)^{\frac{1}{p-1}}\frac{dt}{t}&\le c_{34}\int_{0}^{\infty} \int_{0}^{t/2}(s\vee 1)^{-1-\gamma}ds\left(\frac{\omega(B_t(x))}{t^{N-\alpha p}}\right)^{\frac{1}{p-1}}\frac{dt}{t}\\&\le  c_{35}\int_{0}^{\infty} \left(\frac{\omega(B_t(x))}{t^{N-\alpha p}}\right)^{\frac{1}{p-1}}\frac{dt}{t}
\\&=c_{35}{\bf W}_{\alpha,p}[\mu](x),
\end{align*}

which follows \eqref{31051}.\smallskip

\noindent{\it Step 2: Proof of (\ref{31052})}.   For $t>0$, $r\le t$ and $y\in B_t(x)$ we have $B_r(y)\subset B_{2t}(x)$, thus 
\begin{equation*}
\omega^2_t(B_t(x))=\int _{B_t(x)}H_l\left(\delta_12^{\beta}\left({\bf W}_{\alpha,p}^t[\omega_{B_{2t}(x)}](y)\right)^{\beta}\right)dy.
\end{equation*}
 By Theorem \ref{VHVth2} there exists $c_{36}>0$ such that 
for $0<\delta_1\le c_{36}$, $0<t<2R$, $z\in \mathbb{R}^N$,
\begin{equation}\label{04062}
\int _{B_{4t}(z)}\exp\left(\delta_12^{\beta}\left({\bf W}_{\alpha,p}[\omega_{B_{2t}(z)}](y)\right)^{\beta}\right)dy\le  c_{37}t^{N}.
\end{equation} 
We take $0<\delta_1\le c_{36}$.\smallskip

\noindent{\it Case 1: $x\in B_R$}. If $0<t<2R$, from \eqref{04062} we get
\begin{equation*}
\omega^2_t(B_t(x))\le c_{37}t^{N}\le c_{38}\omega(B_{4t}(x)).
\end{equation*}
If $t\geq 2R$, since for any $|y|\geq 2R$,
\begin{equation*}
{\bf W}_{\alpha,p}[\omega](y)=\int_{|y|/2}^\infty\left(\frac{\omega(B_t(y))}{t^{N-\alpha p}}\right)^{\frac{1}{p-1}}\frac{dt}{t}
\le c_{39}\int_{|y|/2}^\infty t^{-1-\frac{N-\alpha p}{p-1}}dt
\le c_{40}|y|^{-\frac{N-\alpha p}{p-1}},
\end{equation*}
and thanks to \eqref{04062} we have
\begin{align*}
\omega^2_t(B_t(x))&\le  \int _{B_{2R}}\exp\left(\delta_12^{\beta}\left({\bf W}_{\alpha,p}[\omega_{B_{R}}](y)\right)^{\beta}\right)dy+\int _{\mathbb{R}^N\backslash B_{2R}}H_l\left(\delta_12^{\beta}\left({\bf W}_{\alpha,p}[\omega](y)\right)^{\beta}\right)dy\\&\le c_{41}R^N+\int _{\mathbb{R}^N\backslash B_{2R}}H_l\left(c_{42}|y|^{-\frac{\beta(N-\alpha p)}{p-1}}\right)dy
\\&\le  c_{43}+c_{43}\int _{\mathbb{R}^N\backslash B_{2R}}|y|^{-\frac{l\beta(N-\alpha p)}{p-1}}dy
= c_{43}+c_{44}R^{N-\frac{l\beta(N-\alpha p)}{p-1}}\\&\le c_{45}|B_{4t}(x)\cap B_R|\le c_{46}\omega(B_{4t}(x)).
\end{align*}
From this we also have $H_l\left(\delta_1\left({\bf W}_{\alpha,p}[\omega]\right)^{\beta}\right)\in L^1(\mathbb{R}^N)$.\\
{\it Case 2: $x\in \mathbb{R}^N\backslash B_{R}$}.
If $|x|>R+t$ then $\omega^2_t(B_t(x))=0$.
Next we consider the case $R<|x|\le R+t$.
If $0<t<2R$, we have $B_{t/2}((R-\frac{t}{2})\frac{x}{|x|})\subset B_{4t}(x)\cap B_R$; thus from \eqref{04062} we get 
\begin{eqnarray*}
\omega^2_t(B_t(x))\le  c_{47}t^N=c_{48}\left|B_{t/2}\left((R-\frac{t}{2})\frac{x}{|x|}\right)\right|
\le  c_{48}\left|B_{4t}(x)\cap B_R\right|\le  c_{49}\omega(B_{4t}(x)).
\end{eqnarray*}
If $t\geq 2R$, as in Case 1 we also obtain  $\omega^2_t(B_t(x))\le c_{50}\omega(B_{4t}(x))$ since $B_R\subset B_{4t}(x)$. 
Hence, we get \eqref{31052}. Therefore, the result follows with  $\delta_1=\left(\frac{1}{2c_{26}}\left(\frac{\alpha p}{p-1}-1\right)\right) \wedge c_{36}$. 
\end{proof}\medskip


In the next result we obtain estimate on a sequence of solutions of Wolff integral inequations obtained by induction. 

\begin{theorem}\label{TH3} Assume that the assumptions on $\ga$, $p$, $l$, $a$, $\beta$, $\varepsilon$, $f$, $\gm_1$ and $\gm$ of  Theorem \ref{MT5} are fulfilled and $R,K$ are positive real numbers. Suppose that  $\{u_m\}$ is a sequence of nonnegative  measurable  functions in $\mathbb{R}^N$ that satisfies 
  \begin{equation} \label {ite-1}\begin{array}{ll}
  {u_{m + 1}} \leq K {\bf W}_{\alpha ,p}^{R}[P_{l,a,\beta}(u_m ) + \mu ]+f\quad\forall m \in\BBN,\\[2mm]
  \phantom{--}
{u_0} \leq K {\bf W}_{\alpha ,p}^{R}[\mu ]+f.
  \end{array}\end{equation}
 Then there exists $M>0$ depending on $N,\alpha,p,l,a,\beta,\varepsilon, K$ and $R$ such that if 
 \begin{equation*}\label{13061}
 ||{\bf M}_{\alpha p,R}^{\frac{{(p - 1)(\beta  - 1)}}{\beta }}[\mu ]|{|_{{L^\infty }({\mathbb{R}^N})}} \le M~~\text{ and }~ ||{\bf M}_{\alpha p,R}^{\frac{{(p - 1)(\beta  - 1)}}{\beta }}[\mu_1 ]|{|_{{L^\infty }({\mathbb{R}^N})}} \le M,
 \end{equation*}
there holds 
\begin{equation}\label{12321}
    {P_{l,a,\beta}\left({4c_p K {\bf W}_{\alpha ,p}^{R}[\omega_1 ] + 4c_pK{\bf W}_{\alpha,p}^{R}[\omega_2]+ f}\right) \in L^1_{loc}(\mathbb{R}^N)},
    \end{equation}
and 
\begin{equation}\label{1232}
{u_m} \le 2c_p K {\bf W}_{\alpha ,p}^{R}[\omega_1 ] + 2c_pK{\bf W}_{\alpha,p}^{R}[\omega_2]+ f\quad\forall m \in\BBN,
\end{equation}
where 
\begin{equation}\label{omega1}\omega_1=M||{\bf M}_{\alpha p,R}^{\frac{(p-1)(\beta-1)}{\beta}} [1] ||^{-1}_{{L^\infty }(\mathbb{R}^N)}+\mu,
\end{equation}
\begin{equation}\label{omega2}
\omega_2=M||{\bf M}_{\alpha p,R}^{\frac{(p-1)(\beta-1)}{\beta}} [1] ||^{-1}_{{L^\infty }(\mathbb{R}^N)}+\mu_1,
\end{equation}
and  $c_p={1 \vee {4^{\frac{{2 - p}}{{p - 1}}}}}$.\\
Furthermore, if $f\equiv 0$ then \eqref{12321} and \eqref{1232} are satisfied with $\omega_2\equiv 0$.
\end{theorem}
%
\begin{proof} The proof is based upon Theorems \ref{th2} and \ref{th3}. Set $c_{a,\varepsilon}=2\left(1-\left(\frac{a}{a+\varepsilon}\right)^{1/\beta}\right)^{-1}$ and  $\overline{a}=a\left(4c_{a,\varepsilon}c_pK\right)^\beta$.
If $0<M\le 1$ we define $\gw_1$ and $\gw_2$ by (\ref{omega1}) and (\ref{omega2}) respectively.
 We now assume \[||{\bf M}_{\alpha p,R}^{\frac{{(p - 1)(\beta  - 1)}}{\beta }}[\mu ]|{|_{{L^\infty }({\mathbb{R}^N})}} \le M~~\text{ and }~ ||{\bf M}_{\alpha p,R}^{\frac{{(p - 1)(\beta  - 1)}}{\beta }}[\mu_1 ]|{|_{{L^\infty }({\mathbb{R}^N})}} \le M.\]
We prove first that
\begin{equation}\label{06061}
{\bf W}^{R}_{\alpha,p}\left[H_l\left(\overline{a}\left({\bf W}^{R}_{\alpha,p}[\omega_{i}]\right)^\beta\right)\right]
\le   {\bf W}_{\alpha ,p}^{R}[\omega_{i}]~~~\text{ for }~i=1,2.
\end{equation}
By Theorem \ref{th2}, there exist $c,\delta_0>0$ independent on $\mu$ such that $\exp \left( {\delta_0{{\left( { {\bf W}_{\alpha ,p}^R\left[ M^{-1}\omega_{i}  \right]} \right)}^{\beta}}} \right)$ is locally integrable in $\mathbb{R}^N$ and
\begin{equation*}
{ {\bf W}_{\alpha ,p}^R\left[ {\exp \left( {\delta_0{{\left( { {\bf W}_{\alpha ,p}^R\left[  M^{-1}\omega_{i}  \right]} \right)}^{\beta}}} \right)} \right]}\le c  {\bf W}_{\alpha ,p}^R[M^{-1}\omega_{i}] ~~\text{ in } \mathbb{R}^N.
\end{equation*}
Since $\theta^{-l}H_l(s) \le H_l({\theta ^{ - 1}}s)$ for all  $s \geq 0$ and  $0<\theta\le 1$, it follows
\begin{align*}
{\bf W}^{R}_{\alpha,p}\left[M^{-\frac{1}{2}\left(\frac{\beta l}{p-1}+1\right)}H_l\left(\delta_0M^{-\frac{1}{2}\left(\frac{\beta}{p-1}-\frac{1}{l}\right)}\left({\bf W}^{R}_{\alpha,p}[\omega_{i}]\right)^\beta\right)\right]&\le {\bf W}^{R}_{\alpha,p}\left[H_l\left(\delta_0 M^{-\frac{\beta}{p-1}}\left({\bf W}^{R}_{\alpha,p}[\omega_{i}]\right)^\beta\right)\right]\\&\le  {\bf W}^{R}_{\alpha,p}\left[\exp\left(\delta_0 \left({\bf W}^{R}_{\alpha,p}[M^{-1}\omega_{i}]\right)^\beta\right)\right]
\\&\le c M^{-\frac{1}{p-1}} {\bf W}_{\alpha ,p}^R[\omega_{i}].
\end{align*}
Hence, 
\begin{eqnarray*}
{\bf W}^{R}_{\alpha,p}\left[H_l\left(\delta_0M^{-\frac{1}{2}\left(\frac{\beta}{p-1}-\frac{1}{l}\right)}\left({\bf W}^{R}_{\alpha,p}[\omega_{i}]\right)^\beta\right)\right]
\le  c M^{\frac{1}{2(p-1)}\left(\frac{\beta l}{p-1}-1\right)} {\bf W}_{\alpha ,p}^R[\omega_{i}].
\end{eqnarray*}
Therefore \eqref{06061} is achieved if we prove 
\begin{equation*}
\overline{a}\le \delta_0M^{-\frac{1}{2}\left(\frac{\beta}{p-1}-\frac{1}{l}\right)} ~\text{ and }~ c M^{\frac{1}{2(p-1)}\left(\frac{\beta l}{p-1}-1\right)}\le 1,
\end{equation*}
which is equivalent to 
\begin{equation*}
M\le \left(\delta_0\overline{a}^{-1}\right)^{\left(\frac{1}{2}\left(\frac{\beta}{p-1}-\frac{1}{l}\right)\right)^{-1}}\wedge c^{-\left(\frac{1}{2(p-1)}\left(\frac{\beta l}{p-1}-1\right)\right)^{-1}}.
\end{equation*}
Thus, we choose $M=1\wedge \left(\delta_0\overline{a}^{-1}\right)^{\left(\frac{1}{2}\left(\frac{\beta}{p-1}-\frac{1}{l}\right)\right)^{-1}}\wedge c^{-\left(\frac{1}{2(p-1)}\left(\frac{\beta l}{p-1}-1\right)\right)^{-1}}$; we obtain \eqref{06061} and the fact that $H_l\left(\overline{a}\left({\bf W}_{\alpha ,p}^{R}[\omega_{i} ]\right)^\beta \right)\in L^1_{loc}(\mathbb{R}^N)$.\\
Now, we prove \eqref{1232} by induction. Clearly, \eqref{1232} holds with $m=0$. Next we assume that \eqref{1232} holds with $m=n$, and we claim that 
\begin{equation}\label{06062}
{u_{n+1}} \le 2c_p K {\bf W}_{\alpha ,p}^{R}[\omega_1 ] + 2c_pK{\bf W}_{\alpha,p}^{R}[\omega_2]+ f.
\end{equation}
In fact, since \eqref{1232} holds with $m=n$ and  $P_{l,a,\beta}$ is convex, we have
\begin{align*}
{P_{l,a,\beta}\left(u_{n}\right)} &\le P_{l,a,\beta}\left({4c_p K {\bf W}_{\alpha ,p}^{R}[\omega_1 ] + 4c_pK{\bf W}_{\alpha,p}^{R}[\omega_2]+ f}\right)\\&\le P_{l,a,\beta}\left(4c_{a,\varepsilon}c_p K {\bf W}_{\alpha ,p}^{R}[\omega_1 ]\right) + P_{l,\varepsilon,a}\left(4c_{a,\varepsilon}c_p K {\bf W}_{\alpha ,p}^{R}[\omega_2]\right)+P_{l,a,\beta}\left(\left(1+\frac{\varepsilon}{a}\right)^{1/\beta}f\right)\\&= H_l\left(\overline{a}\left({\bf W}_{\alpha ,p}^{R}[\omega_1 ]\right)^\beta \right)+H_l\left(\overline{a}\left({\bf W}_{\alpha ,p}^{R}[\omega_2 ]\right)^\beta \right)+P_{l,a+\varepsilon,\beta}(f).
\end{align*} 
From this we derive \eqref{12321}. By the definition of $u_{n+1}$ and the sub-additive property of ${\bf W}^{R}_{\alpha,p}[.]$, we obtain 
\begin{align*}
{u_{n + 1}} &\leq K {\bf W}_{\alpha ,p}^{R}\left[H_l\left(\overline{a}\left({\bf W}_{\alpha ,p}^{R}[\omega_1 ]\right)^\beta \right)+H_l\left(\overline{a}\left({\bf W}_{\alpha ,p}^{R}[\omega_2 ]\right)^\beta \right)+P_{l,a+\varepsilon,\beta}(f) +\mu\right]+f\\&\le   c_pK {\bf W}_{\alpha ,p}^{R}\left[H_l\left(\overline{a}\left({\bf W}_{\alpha ,p}^{R}[\omega_1 ]\right)^\beta \right) \right]+ c_pK {\bf W}_{\alpha ,p}^{R}\left[H_l\left(\overline{a}\left({\bf W}_{\alpha ,p}^{R}[\omega_2 ]\right)^\beta \right) \right]\\&~\phantom{------} +c_pK {\bf W}_{\alpha ,p}^{R}\left[P_{l,a+\varepsilon,\beta}(f) \right]+c_pK {\bf W}_{\alpha ,p}^{R}\left[\mu\right]+f.
\end{align*}
Hence follows \eqref{06062} from \eqref{06061}. This completes the proof of the theorem.
\end{proof}\\

The next result is obtained by an easy adaptation of the proof Theorem \ref{TH3}.

\begin{theorem}\label{TH3-B}
Assume that the assumptions on $\ga$, $p$, $a$, $l$, $\beta$, $\varepsilon$, $f$, $\gm_1$ and $\gm$ of  Theorem \ref{MT6} are fulfilled and $R,K$ are positive real numbers. Suppose that  $\{u_m\}$ is a sequence of nonnegative measurable functions in $\mathbb{R}^N$ that satisfies 
  \begin{equation} \label {ite-1}\begin{array}{ll}
  {u_{m + 1}} \leq K {\bf W}_{\alpha ,p}[P_{l,a,\beta}(u_m ) + \mu ]+f\quad\forall m \in\BBN,\\[2mm]
  \phantom{--}
{u_0} \leq K {\bf W}_{\alpha ,p}[\mu ]+f.
  \end{array}\end{equation}
   Then there exists $M>0$ depending on $N,\alpha,p,l,a,\beta,\varepsilon, K$ and $R$  such that if 
  \begin{equation*}\label{13062}
  ||{\bf M}_{\alpha p}^{\frac{{(p - 1)(\beta  - 1)}}{\beta }}[\mu ]|{|_{{L^\infty }({\mathbb{R}^N})}} \le M~~\text{ and }~||{\bf M}_{\alpha p}^{\frac{{(p - 1)(\beta  - 1)}}{\beta }}[\mu_1 ]|{|_{{L^\infty }({\mathbb{R}^N})}} \le M,
  \end{equation*}
there holds
\begin{equation}\label{12320}
    {P_{l,a,\beta}\left( {{{ {4c_p K {\bf W}_{\alpha ,p}[\omega_3 ] + 4c_pK{\bf W}_{\alpha,p}[\omega_4]+ f} }}\;} \right) \in L^1(\mathbb{R}^N)},
    \end{equation}
    and 
\begin{equation}\label{1233}
{u_m} \le 2c_p K {\bf W}_{\alpha ,p}[\omega_3 ] + 2c_pK{\bf W}_{\alpha,p}[\omega_4]+ f\quad\forall m \in\BBN,
\end{equation}
where 
\begin{equation}\label{omega3}
\omega_3=M||{\bf M}_{\alpha p}^{\frac{(p-1)(\beta-1)}{\beta}} [\chi_{B_{R}}] ||^{-1}_{{L^\infty }(\mathbb{R}^N)}\chi_{B_{R}}+\mu,    \end{equation}
and
\begin{equation}\label{omega4}\omega_4=M||{\bf M}_{\alpha p}^{\frac{(p-1)(\beta-1)}{\beta}} [\chi_{B_{R}}] ||^{-1}_{{L^\infty }(\mathbb{R}^N)}\chi_{B_{R}}+\mu_1.
\end{equation}  
Furthermore, if $f\equiv 0$ then \eqref{12320} and \eqref{1233} are satisfied with $\omega_4\equiv 0$.
\end{theorem}

Let $P\in C(\mathbb{R}^+)$ be a decreasing positive function. The $(\alpha,P)$-Orlicz-Bessel capacity of a Borel set $E\subset\mathbb{R}^N$ is defined by (see \cite[Sect 2.6]{AH})
\[\text{Cap}_{{\mathbf{G}_\alpha },P}(E) = \inf \left\{ {\int_{\mathbb{R}^N} {P(f)} :{\mathbf{G}_\alpha }\ast f \geq {\chi _E},f \geq 0,P(f) \in {L^1}({\mathbb{R}^N})} \right\},\]
and the $(\alpha,P)$-Orlicz-Riesz capacity
\[\text{Cap}_{{\mathbf{I}_\alpha },P}(E) = \inf \left\{ {\int_{\mathbb{R}^N} {P(f)} :{\mathbf{I}_\alpha }\ast f \geq {\chi _E},f \geq 0,P(f) \in {L^1}({\mathbb{R}^N})} \right\}.\]
\begin{theorem}\label{TH1} Let $\alpha>0$, $p>1$, $a>0$, $c>0$, $l\in \mathbb{N}^*$ and $\beta\geq 1$ such that $l\beta>p-1$ and $0<\alpha p <N$. Let $\mu\in \mathfrak{M}^+(\mathbb{R}^N)$.\\
{\bf 1}. Let $0<R\le \infty$. If $u$ is a   nonnegative Borel function in $\mathbb{R}^N$ such that  $P_{l,a,\beta}(u)$ is locally integrable in $\mathbb{R}^N$ and
\begin{equation}
u(x)\geq c {\bf W}^R_{\alpha,p}[P_{l,a,\beta}(u)+\mu](x)\qquad\forall x\in \mathbb{R}^N,
\end{equation}
then the following statements holds. \\
(i) If $R<\infty$, there exists a positive constant $C_1$ depending on $N,\alpha,p,l,a,\beta,c$ and $R$ such that
\begin{equation}
\label{06064}\int_E P_{l,a,\beta}(u)dx+\mu(E)  \le C_1\text{Cap}_{{\mathbf{G}_{\alpha p}},Q_p^*}(E)\qquad\forall E\subset \mathbb{R}^N,\, E\text{ Borel}.
\end{equation} 
(ii) If $R=\infty$, there exists a positive constant $C_2$ depending on $N,\alpha,p,l,a,\beta,c$ such that
\begin{equation}
\label{06065}\int_E P_{l,a,\beta}(u)dx+\mu(E)  \le C_2\text{Cap}_{{\mathbf{I}_{\alpha p}},Q_p^*}(E)\qquad\forall E\subset \mathbb{R}^N,\, E\text{ Borel}.
\end{equation} 
{\bf 2}. Let $\Omega$ be a bounded domain in $\mathbb{R}^N$, $\mu\in \mathfrak{M}^+(\Omega)$ and $\delta\in (0,1)$. If $u$ is a nonnegative Borel function in $\Omega$ such that  $P_{l,a,\beta}(u)$ is locally integrable in  $\Omega$ and 
\begin{equation}\label{25032}
u(x)\geq c{\bf W}^{\delta d(x,\partial\Omega)}_{\alpha,p}[P_{l,a,\beta}(u)+\mu](x)\qquad\forall x\in \Omega,
\end{equation}
then,  for any compact set $K\subset \Omega$, there exists a positive constant $C_3$ depending on $N,\alpha,p,l,a,\beta,c, \delta$ and $dist(K,\partial\Omega)$ such that 
\begin{equation}\label{08063}
\int_E {P_{l,a,\beta}(u)dx}  + \mu (E) \leq C_3 \text{Cap}_{{\mathbf{G}_{\alpha p}},Q_p^*}(E)\quad\forall E\subset K,E\textrm{  Borel},
\end{equation}
where  $Q_p^*$ is the complementary function to $Q_p$.\\

\end{theorem}
\begin{proof} Set $d\omega=P_{l,a,\beta}(u)dx+d\mu$.\\
\textbf{1.} We have $$P_{l,a,\beta}\left(c{\bf W}^{R}_{\alpha,p}[\omega]\right)dx\le d\omega ~~\textrm{ in } \mathbb{R}^N.$$
Let $M_\omega$ denote the centered Hardy-Littlewood maximal function which is defined for any $f\in L_{loc}^1(\BBR^N,d\gw)$  by 
\[{M_\omega }f(x) = \sup _{t > 0} \frac{1}{\omega (B_t(x))}\myint{B_t(x)}{}|f|d\omega  .\]
If $E\subset\BBR^N$ is a Borel set, we have
\[\myint{\mathbb{R}^N}{} {{{\left( {{M_{\omega} }{\chi _E}} \right)}^{\frac{{l\beta }}{{p - 1}}}}P_{l,a,\beta}\left( { c{\bf W}_{\alpha ,p}^R[\omega ]} \right)dx}  \le \myint{\mathbb{R}^N}{} {{{\left( {{M_{\omega} }{\chi _E}} \right)}^{\frac{{l\beta }}{{p - 1}}}}d\omega }.\]
Since $M_\omega$ is bounded on $L^s(\mathbb{R}^N,d\omega)$, $s>1$, we deduce from Fefferman's result \cite{Fe} that
\[\int_{\mathbb{R}^N} {{{\left( {{M_{\omega} }{\chi _E}} \right)}^{\frac{{l\beta }}{{p - 1}}}}P_{l,a,\beta}\left( { c{\bf W}_{\alpha ,p}^R[\omega ]} \right)dx}  \le c_{51}{\omega }(E),\]
for some constant $c_{51}$ only depends on $N$ and $\frac{l\beta}{p-1}$.  Since ${{M_\omega }{\chi _E}}\le 1$, we derive 
\begin{eqnarray*}
 {\left( {{M_\omega }{\chi _E}(x)} \right)^{\frac{{l\beta }}{{p - 1}}}}P_{l,a,\beta}\left( c{\bf W}_{\alpha ,p}^{R}[\omega ](x) \right) &\geq& P_{l,a,\beta}\left(c\left( M_\omega \chi _E(x)\right)^{\frac{1}{p - 1}} {\bf W}_{\alpha ,p}^{R}[\omega ](x) \right) \\ 
  &\geq& P_{l,a,\beta}\left(  c{\bf W}_{\alpha ,p}^{R}[\omega _E](x) \right),
\end{eqnarray*}
where $\omega_E=\chi_E\omega$. Thus
 \begin{equation}\label{06063}
\int_{\mathbb{R}^N} {P_{l,a,\beta}\left( { c{\bf W}_{\alpha ,p}^R[\omega_E ]} \right)dx}  \le c_{51}{\omega }(E)\qquad\forall E \subset \mathbb{R}^N, E\text{ Borel}.
\end{equation}
From \eqref{ine0}, \eqref{ine2} and \eqref{ine3} we get
\[\int_{\mathbb{R}^N} P_{l,a,\beta}\left( c {\bf W}_{\alpha ,p}^{R}[\omega _E](x) \right)dx  \geq \int_{\mathbb{R}^N} Q_p\left( c_{52}{\bf G}_{\alpha p}[\omega _E](x) \right)dx~~\text{ if } R<\infty,\] 
and 
\[\int_{\mathbb{R}^N} P_{l,a,\beta}\left( c {\bf W}_{\alpha ,p}^{R}[\omega _E](x) \right)dx  \geq \int_{\mathbb{R}^N} Q_p\left( c_{53}{\bf I}_{\alpha p}[\omega _E](x) \right)dx~~\text{ if } R=\infty,\] 
where $Q_p$ is defined by $(\ref{EQ13})$ and $c_{52} = (c_2\beta )^{ - 1}a^{\frac{p - 1}{\beta }}c^{p - 1}$ if $p\not= 2$, $c_{52}= c_3^{ - 1}a^{\frac{1}{\beta }}c $ if $p=2$ (the constants $c_2,c_3$ defined in \eqref{ine2} and \eqref{ine3}, depend on $R$, therefore $c_{52}=c_{52}(r_K)$) and $c_{53} = (c_1\beta )^{ - 1}a^{\frac{p - 1}{\beta }}c^{p - 1}$ if $p\not= 2$, $c_{53}= a^{\frac{1}{\beta }}c $ if $p=2$.
Thus, from \eqref{06063} we obtain that for all Borel set $E\subset\mathbb{R}^N$ there holds
 \[\int_{{\mathbb{R}^N}} {{Q_p}\left( {{c_{52}}{{\bf G}_{\alpha p}}[{\omega _E}](x)} \right)dx}  \le {c_{51}}\omega (E) ~\text{if } R<\infty,\]
 and 
  \[\int_{{\mathbb{R}^N}} {{Q_p}\left( {{c_{53}}{{\bf I}_{\alpha p}}[{\omega _E}](x)} \right)dx}  \le {c_{51}}\omega (E) ~\textrm{if } R=\infty.\]
We recall that  $Q_p^*(s)=\sup_{t>0}\{st-Q_p(t)\}$ satisfies the sub-additivity $\Delta_2$-condition (see Chapter 2 in \cite{ReRa}).\\
(i) We assume $R<\infty$.
 For every $f\geq 0$, $Q^*_p(f)\in L^1(\Omega)$ such that ${\bf G}_{\alpha p}\ast f\geq \chi_E$, we have 
 \begin{align*}
  \omega (E) &\le \int_{\mathbb{R}^N} {{ \bf G}_{\alpha p}}\ast fd{\omega _E}  = {(2{c_{51}})^{ - 1}}\int_{\mathbb{R}^N} {\left( {{c_{52}}{{\bf G}_{\alpha p}}\left[ {{\omega _E}} \right]} \right)\left( {2{c_{51}}c_{52}^{ - 1}f} \right)dx}  \\ 
   &\le{(2{c_{51}})^{ - 1}}\int_{\mathbb{R}^N} {{Q_p}\left( {{c_{52}}{{\bf G}_{\alpha p}}\left[ {{\omega _E}} \right]} \right)dx}  + {(2{c_{51}})^{ - 1}}\int_{\mathbb{R}^N} {Q_p^*\left( {2{c_{51}}c_{52}^{ - 1}f} \right)dx}  \\ 
   &\le {2^{ - 1}}\omega (E) + {c_{54}}\int_{\mathbb{R}^N} {Q_p^*\left( f \right)dx},
 \end{align*}
the last inequality following from the $\Delta_2$-condition.
Notice that $c_{54}$, as well as the next constant $c_{55}$, depends on $r_K$. Thus, \[\omega (E) \leq 2c_{54}\int_{\mathbb{R}^N} {Q_p^*\left( f \right)dx}. \]
Then, we get \[\omega (E) \le {c_{55}}\text{Cap}_{{\mathbf{G}_{\alpha p}},Q_p^*}(E)\qquad\forall E \subset \mathbb{R}^N, E\text{ Borel}.\]
Which implies \eqref{06064}.\\
(ii) We assume $R=\infty$. For every $f\geq 0$, $Q^*_p(f)\in L^1(\Omega)$ such that ${ \bf I}_{\alpha p}\ast f\geq \chi_E$, since $\mathbf{I}_{\alpha p}*\omega_E=\mathbf{I}_{\alpha p}[\omega_E]$, as above we have 
\begin{align*}
 \omega (E) &\le \int_{\mathbb{R}^N} {{\bf  I}_{\alpha p}}\ast fd{\omega _E} = \int_{\mathbb{R}^N} {\left( {{\bf I}_{\alpha p}}* {{\omega _E}} \right)fdx} = \int_{\mathbb{R}^N} {{{\bf I}_{\alpha p}}\left[ {{\omega _E}} \right]fdx}   \\ 
  &\le {2^{ - 1}}\omega (E) + {c_{56}}\int_{\mathbb{R}^N} {Q_p^*\left( f \right)dx},
\end{align*}
Then, it follows \eqref{06065}.\\
\textbf{ 2.}  Let $K\subset\Omega$ be compact. Set $r_K=dist(K,\partial \Omega)$ and $\Omega_K=\{x\in\Omega: d(x,K)<r_K/2\}$. We have 
$$P_{l,a,\beta}\left(c{\bf W}^{\delta d(x,\partial\Omega)}_{\alpha,p}[\omega]\right)dx\le d\omega ~~\textrm{ in } \Omega.$$
Thus, for any Borel set $E\subset K$  
\[\int_\Omega  {{{\left( {{M_\omega }{\chi _E}} \right)}^{\frac{{l\beta }}{{p - 1}}}}P_{l,a,\beta}\left(  c{\bf W}_{\alpha ,p}^{\delta d(x,\partial \Omega )}[\omega ] \right)dx}  \le \int_\Omega  {{{\left( {{M_\omega }{\chi _E}} \right)}^{\frac{{l\beta }}{{p - 1}}}}d\omega }. \]
As above we get
\begin{equation}\label{25033}
\int_\Omega  P_{l,a,\beta}\left( c {\bf W}_{\alpha ,p}^{\delta d(x,\partial \Omega )}[\omega _E](x) \right)dx  \leq {c_{51}}\omega (E)\quad\forall E\subset K, E\textrm{ Borel}.
\end{equation}
Note that if $x\in \Omega$ and $d(x,\partial \Omega)\le r_K/8$, then $B_t(x)\subset \Omega\backslash \Omega_K$ for all $t\in (0,\delta d(x,\partial\Omega))$;
 indeed, for all $y\in B_t(x)$\\
\[d(y,\partial \Omega ) \le d(x,\partial \Omega ) + |x - y| < (1 + \delta )d(x,\partial \Omega ) < \frac{1}{4}{r_K},\]
thus
\[d(y,K) \geq d(K,\partial \Omega ) - d(y,\partial \Omega ) > \frac{3}{4}{r_K} > \frac{1}{2}{r_K},\]
which implies $y\notin \Omega_K$. 
We deduce that  
\[ {\bf W}_{\alpha ,p}^{\delta d(x,\partial \Omega )}[{\omega _E}](x) \geq  {\bf W}_{\alpha ,p}^{\frac{\delta }{8}{r_K}}[{\omega _E}](x)\qquad\forall  x\in \Omega,\]
and 
\[ {\bf W}_{\alpha ,p}^{\frac{\delta }{8}{r_K}}[{\omega _E}](x) = 0\qquad\forall  x\in \Omega^c.\]
Hence we obtain from \eqref{25033},
\begin{equation}\label{ine4}\int_{\mathbb{R}^N}  P_{l,a,\beta}\left( c{\bf W}_{\alpha ,p}^{\frac{\delta }{8}{r_K}}[{\omega _E}](x)\right)dx  \le {c_{51}}\omega (E)\qquad\forall  E\subset K,\, E\text{ Borel}.
\end{equation}
As above we also obtain  \[\omega (E) \le {c_{57}}\text{Cap}_{{\mathbf{G}_{\alpha p}},Q_p^*}(E)\qquad\forall  E\subset K,\, E\text{ Borel},\]
where the positive constant $c_{57}$ depends on $r_K$. Inequality  \eqref{08063} follows and this completes the proof of the Theorem.
 \end{proof}\\
 
\noindent\begin{proof}[Proof of Theorem \ref{MT5}] Consider  the sequence $\{u_m\}_{m\geq 0}$ of nonnegative functions defined by $u_0=f$ and 
$$u_{m+1}=  {\bf W}_{\alpha ,p}^{R}[P_{l,a,\beta}(u_m )] + f~~\textrm{ in }~ \mathbb{R}^N\quad\forall m\geq 0.$$
By Theorem \ref{TH3}, there exists $M>0$ depending on $N,\alpha,p,l,a,\beta,\varepsilon$ and $R$ such that if \eqref{14061} holds, then  $\{u_m\}_{m\geq 0}$ is well defined and  \eqref{12321} and \eqref{1232} are satisfied.  It is easy to see that $\{u_m\}$ is nondecreasing. Hence, thanks to the dominated convergence theorem, we obtain that  $ u(x) = \mathop {\lim }\limits_{m \to \infty } {u_m}(x)$
is a solution of equation \eqref{MT5a} which satisfies \eqref{MT5b}.\\   
Conversely, we obtain \eqref{MT5c} directly from Theorem \ref{TH1}, Part 1, (i).
 \end{proof}\medskip\\
 
 \noindent\begin{proof}[Proof of Theorem \ref{MT6}] The proof is similar to the previous one by using Theorem \ref{TH3-B} and  Theorem \ref{TH1}, Part 1, (ii).
  \end{proof}

\section{Quasilinear Dirichlet problems}
Let $\Omega$ be a bounded domain in $\mathbb{R}^N$.   If $\mu\in\mathfrak{M}_b(\Omega)$, we denote by $\mu^+$ and $\mu^-$ respectively its positive and negative parts in the Jordan decomposition. We denote by $\mathfrak{M}_0(\Omega)$ the space of measures in $\Omega$ which are absolutely continuous with respect to the $c^{\Omega}_{1,p}$-capacity defined on a compact set $K\subset\Omega$ by
 \begin{equation*}
c^{\Omega}_{1,p}(K)=\inf\left\{\int_{\Omega}{}|{\nabla \varphi}|^pdx:\varphi\geq \chi_K,\varphi\in C^\infty_c(\Omega)\right\}.
 \end{equation*}
 We also denote $\mathfrak{M}_s(\Omega)$ the space of measures in $\Omega$ with support on a set of zero $c^{\Omega}_{1,p}$-capacity. Classically, any $\mu\in\mathfrak{M}_b(\Omega)$ can be written in a unique way under the form $\mu=\mu_0+\mu_s$ where $\mu_0\in \mathfrak{M}_0(\Omega)\cap \mathfrak{M}_b(\Omega)$ and $\mu_s\in \mathfrak{M}_s(\Omega)$.
It is well known  that any  $\mu_0\in \mathfrak{M}_0(\Omega)\cap\mathfrak{M}_b(\Omega)$ can be written under the form $\mu_0=f-div ~g$ where $f\in L^1(\Omega)$ and $g\in L^{p'}(\Omega,\mathbb{R}^N)$.
 
 For $k>0$ and $s\in\mathbb{R}$ we set $T_k(s)=\max\{\min\{s,k\},-k\}$. If $u$ is a measurable function defined  in $\Omega$, finite a.e. and such that $T_k(u)\in W^{1,p}_{loc}(\Omega)$ for any $k>0$, there exists a measurable function $v:\Omega\to \mathbb{R}^N$ such that $\nabla T_k(u)=\chi_{|u|\leq k}v$ 
 a.e. in $\Omega$ and for all $k>0$. We define the gradient $\nabla u$ of $u$ by $v=\nabla u$. We recall the definition of a renormalized solution given in \cite{DMOP}.
 
 \begin{definition} Let $\mu=\mu_0+\mu_s\in\mathfrak{M}_b(\Omega)$. A measurable  function $u$ defined in $\Omega$ and finite a.e. is called a renormalized solution of 
 
\begin{equation}
\label{pro1}\begin{array}{ll}
  - {\Delta _p}u = \mu \qquad&\;in\;\Omega,  \\ 
 \phantom{ - {\Delta _p}}u = 0~~~&\;on\;\partial \Omega,  \\ 
 \end{array}
\end{equation}
  if $T_k(u)\in W^{1,p}_0(\Omega)$ for any $k>0$, $|{\nabla u}|^{p-1}\in L^r(\Omega)$ for any $0<r<\frac{N}{N-1}$, and $u$ has the property that for any $k>0$ there exist $\lambda_k^+$ and $\lambda_k^-$ belonging to $\mathfrak{M}_{b}^+\cap\mathfrak{M}_0(\Omega)$, respectively concentrated on the sets $u=k$ and $u=-k$, with the property that 
  $\mu_k^+\rightharpoonup\mu_s^+$, $\mu_k^-\rightharpoonup\lambda_s^-$ in the narrow topology of measures and such that
 \[
 \int_{\{|u|<k\}}\left\vert \nabla u\right\vert ^{p-2}\nabla u.\nabla\varphi
 dx=\int_{\{|u|<k\}}{\varphi d}{\mu_{0}}+\int_{\Omega}\varphi d\lambda_{k}%
 ^{+}-\int_{\Omega}\varphi d\lambda_{k}^{-},%
 \]
  for every $\varphi\in W^{1,p}_0(\Omega)\cap L^{\infty}(\Omega)$.

 \end{definition}

 \begin{remark}
 \label{R1} We recall that if $u$ is a renormalized solution to problem \eqref{pro1}, then $\frac{|\nabla u|^p}{(|u|+1)^r}\in L^1(\Omega)$ for all $r>1$. From this it follows by H\"older's inequality that $u\in W^{1,p_1}_{0}(\Omega)$ for all $1\le p_1<p$ provided $e^{a|u|}\in L^1(\Omega)$ for some $a>0$. Furthermore, $u\geq 0$ $a.e.$ in $\Omega$ if $\mu\in\mathfrak{M}_{b}^+(\Omega)$.
 \end{remark}

The following general stability result has been proved in \cite[Th 4.1]{DMOP}.
%

 \begin{theorem}
 \label{P3} Let $\mu=\mu_{0}+\mu_{s}^{+}-\mu_{s}^{-},$ with $\mu_{0}%
 =F-\operatorname{div}g\in\mathfrak{M}_{0}(\Omega)$ and  $\mu_{s}^{+}$, $\mu ^{-}_{s}$
belonging to $\mathfrak{M}_{s}^{+}(\Omega).$ Let
$ \mu_{n}=F_{n}-\operatorname{div}g_{n}+\rho_{n}-\eta_{n}$ with
 $F_{n}\in L^{1}(\Omega)$, $g_{n} \in(L^{p^{\prime}}(\Omega))^{N}$ and $\rho_{n}$, $\eta
 _{n}$ belonging to $\mathfrak{M}_{b}^{+}(\Omega)$.
 Assume that $\{F_{n}\}$ converges to $F$ weakly in $L^{1}(\Omega)$, $\{g_{n}\}$
 converges to $g$ strongly in $(L^{p^{\prime}}(\Omega))^{N}$ and
 $(\operatorname{div}g_{n})$ is bounded in $\mathfrak{M}_{b}(\Omega)$; assume also that
 $\{\rho_{n}\}$ converges to $\mu_{s}^{+}$ and $\{\eta_{n}\}$ to $\mu
 _{s}^{-}$ in the narrow topology. If $\{u_{n}\}$ is a sequence of renormalized solutions of \eqref{pro1} with data $\mu_n$, 
 then, up to a subsequence, it converges a.e. in $\Omega$ to a
 renormalized solution $u$ of problem (\ref{pro1}). Furthermore, $T_{k}(u_{n})$ converges to $T_{k}(u)$  in $W_{0}^{1,p}%
 (\Omega)$ for any $k>0$.
 \end{theorem}
 
We also recall the following estimate \cite[Th 2.1]{PhVe}.
 \begin{theorem}\label{TH4}
 Let $\Omega$ be a bounded domain of $\mathbb{R}^N$. Then there exists a constant $K_1>0$, depending on $p$ and $N$ such that if $\mu\in \mathfrak{M}^+_b(\Omega)$ and $u$ is a nonnegative renormalized solution of problem (\ref{pro1}) with  data $\mu$, there holds
    \begin{equation}
        \frac{1}{K_1}{\bf W}^{\frac{d(x,\partial\Omega)}{3}}_{1,p}[\mu](x) \le u(x)\leq K_1 {\bf W}^{2\,diam\,(\Omega)}_{1,p}[\mu](x)~~ \forall x\in \Omega,
        \end{equation}
 where the positive constant $K_1$ only depends on $N,p$.  
 \end{theorem} 
\begin{proof}[Proof of Theorem \ref{MT1}] Let $\{u_m\}_{m\in \BBN}$ be a 
sequence of nonnegative renormalized solutions of the following problems  
\begin{equation*}
\begin{array}{ll}
  - {\Delta _p}u_0 = \mu\qquad&\text{in }\,\Omega,  \\ \phantom{  - {\Delta _p}}
 u_0 = 0\qquad&\text{on }\,\partial\Omega, 
 \end{array}
\end{equation*}
 and, for $m\in\BBN$,
\begin{equation*}
\begin{array}{ll}
  - {\Delta _p}u_{m+1}  =P_{l,a,\beta}(u_{m})+ \mu\qquad&\text{in }\,\Omega,  \\ \phantom{  - {\Delta _p}}
 u_{m+1}=0\qquad&\text{on }\,\partial\Omega. 
 \end{array} 
\end{equation*}
Clearly, we can assume that $\{u_m\}$ is nondecreasing, see \cite{PhVe2}.  
By Theorem \ref{TH4} we have 
\begin{equation*}\begin{array}{ll}
    \chi_{\Omega}u_0\leq K_1 {\bf W}^{R}_{1,p}[\mu],\\[2mm]
    \!\!\!\!\!\!\!\chi_{\Omega}u_{m+1}\leq K_1 {\bf W}^{R}_{1,p}[P_{l,a,\beta}(u_{m})+\mu]~~\forall m \in 
    \mathbb{N},
  \end{array}  \end{equation*}
where $R=2\,diam\,(\Omega)$. Thus, by Theorem \ref{TH3} with $f\equiv 0$,
 there exists $M>0$ depending on $N,p,l,a,\beta,K_1$ and $R$ such that  $P_{l,a,\beta}(4c_pK_1{\bf W}_{1 ,p}^R[\omega])\in L^1(\Omega)$ and
 \begin{equation}\label{1335}
    {u_m(x)} \le 2c_pK_1 {\bf W}_{1 ,p}^R[\omega](x) \quad\forall x\in \Omega, m\in\BBN,
    \end{equation}
    provided that\[||M_{ p,R}^{\frac{{(p - 1)(\beta  - 1)}}{\beta }}[\mu ]|{|_{{L^\infty }({\mathbb{R}^N})}} \le M,\]
 where $\omega=M||{\bf M}_{ p,R}^{\frac{(p-1)(\beta-1)}{\beta}} [1] ||^{-1}_{{L^\infty }(\mathbb{R}^N)}+\mu$ and $c_p={1 \vee {4^{\frac{{2 - p}}{{p - 1}}}}} $.   
This implies that $\{u_m\}$ is well defined and nondecreasing. Thus  $\{u_m\}$ converges a.e in $\Omega$ to some function $u$ which satisfies \eqref{MT1b}  in $\Omega$. Furthermore, we deduce  from \eqref{1335} and the monotone convergence theorem that $P_{l,a,\beta}(u_{m})\to P_{l,a,\beta}(u)$ in $L^1(\Omega)$. Finally, by Theorem \ref{P3} we obtain that $u$ is a renormalized solution of \eqref{MT1a}.\\
     Conversely, assume that \eqref{MT1a} admits a nonnegative renormalized solution $u$. By Theorem \ref{TH4} there holds 
   $$u(x)\geq \frac{1}{K_1}{\bf W}^{\frac{d(x,\partial\Omega)}{3}}_{1,p}[P_{l,a,\beta}(u)+\mu](x)~~\textrm{ for all } x\in \Omega.$$
  Hence,  we achieve \eqref{MT1c} from Theorem \ref{TH1}, Part 2.
\end{proof}\\

\noindent{\bf Applications.} We consider the case $p=2$, $\beta=1$. Then $l=2$ and
$$P_{l,a,\beta}(r)=e^{ar}-1-ar.
$$
 If $\Omega$ is a bounded domain in $\BBR^N$, there exists $M>0$ such that if $\mu$ is a positive Radon measure in $\Omega$ which satisfies
$$
\mu (B_t(x))\leq M
t^{N-2}\qquad\forall t>0 \text{ and }\text{ almost all  }\,x\in\Omega,
$$
there exists a positive solution $u$ to the following problem
$$\BA {ll}
-\Gd u=e^{au}-1-au+\mu \qquad&\text{ in  }\;\Omega,\\
\phantom{-\Gd}
u=0\qquad&\text{ on  }\;\partial\Omega.
\EA$$
 Furthermore
 $$
u(x)\leq  K(N)\int_0^{2\,diam\,\Gw}\frac{\omega (B_t(x))}{t^{N-1}}dt=K(N)\int_0^{2\,diam\,(\Gw)}\frac{\mu (B_t(x))}{t^{N-1}}dt+b\qquad\forall x\in\Omega.
$$
where $b=2K(N)M||{\bf M}_{ 2,2\,diam\,(\Omega)} [1] ||_{{L^\infty }(\mathbb{R}^N)}^{-1}|B_1|(diam\,\Gw)^2$.
In the case $N=2$ this result has already been proved by Richard and V\'eron \cite[Prop 2.4]{RV}.
\section{p-superharmonic functions and quasilinear equations in $\mathbb{R}^N$ }
We recall some definitions and properties   of $p$-superharmonic functions. \\

\begin{definition} A function $u$ is said to be $p$-harmonic in $\mathbb{R}^N$ if $u\in W^{1,p}_{loc}(\mathbb{R}^N)\cap C(\mathbb{R}^N)$ and $-\Delta_p u=0$ in $\mathcal{D}'(\mathbb{R}^N)$.
A function $u$ is called a $p$-supersolution in $\mathbb{R}^N$ if $u\in W^{1,p}_{loc}(\mathbb{R}^N)$ and $-\Delta_p u\geq 0$ in $\mathcal{D}'(\mathbb{R}^N)$. 
\end{definition}
\begin{definition}
A lower semicontinuous (l.s.c) function $u: \mathbb{R}^N \to (-\infty,\infty]$ is called $p$-super-\\harmonic if $u$ is not identically infinite  and if, for all open $D\subset \subset \mathbb{R}^N$ and all $v\in C(\overline{D})$, $p$-harmonic in $D$, $v\le u$ on $\partial D$ implies $v\le u$ in $D$.
\end{definition}
Let $u$ be a $p$-superharmonic in $\mathbb{R}^N$. It is well known that $u\wedge k\in W^{1,p}_{loc}(\mathbb{R}^N)$ is a p-supersolution for all $k>0$ and $u<\infty$ a.e in $\mathbb{R}^N$, thus, $u$ has a gradient (see the previous section). We also have $|\nabla u|^{p-1}\in L^q_{loc}(\mathbb{R}^N)$, $\frac{|\nabla u|^p}{(|u|+1)^r}\in L^1_{loc}(\mathbb{R}^N)$ and $u\in L^s_{loc}(\mathbb{R}^N)$ for $1\le q<\frac{N}{N-1}$ and $r>1$, $1\le s < \frac{N(p-1)}{N-p}$ (see \cite[Theorem 7.46]{HeKiMa}). In particular, if $e^{a|u|}\in L^1_{loc}(\mathbb{R}^N)$ for some $a>0$, then $u\in W^{1,p_1}_{loc}(\mathbb{R}^N)$ for all $1\le p_1<p$  by H\"older's inequality. Thus for any  $0\le \varphi \in C^1_c(\Omega)$, by the dominated convergence theorem, 
\begin{equation*}
\left\langle -\Delta_p u,\varphi \right\rangle =\int_{\mathbb{R}^N}|\nabla u|^{p-2}\nabla u\nabla \varphi dx
=\mathop {\lim }\limits_{k \to \infty }\int_{\mathbb{R}^N}|\nabla (u\wedge k)|^{p-2}\nabla (u\wedge k) \nabla \varphi \geq 0.\end{equation*}
Hence, by the Riesz Representation Theorem we conclude that there is a nonnegative Radon measure denoted by $\mu[u]$, called Riesz measure,  such that $-\Delta_p u=\mu[u]$ in $\mathcal{D}'(\mathbb{R}^N)$.\\

The following weak convergence result for  Riesz measures proved in \cite{TW4} will be used to prove the existence of $p$-superharmonic solutions to quasilinear equations. 
\begin{theorem}\label{07062}
Suppose that $\{u_n\}$  is a sequence of nonnegative $p$-superharmonic functions in $\mathbb{R}^N$  that converges a.e to a $p$-superharmonic function $u$. Then the sequence of measures $\{\mu[u_n]\}$ converges to $\mu[u]$ in the weak sense of measures.
\end{theorem} 
The next theorem is proved in \cite{PhVe}
\begin{theorem}\label{07061}
Let $\mu$ be  a measure in $\mathfrak{M}^+(\mathbb{R}^N)$. Suppose that ${\bf W}_{1,p}[\mu]<\infty$ a.e. Then there exists a nonnegative $p$-superharmonic function $u$ in $\mathbb{R}^N$ such that  $-\Delta_p u=\mu $ in $\mathcal{D}'(\mathbb{R}^N)$, $\inf_{\mathbb{R}^N}u=0$ and
\begin{equation}\label{07063}
\frac{1}{K_1}{\bf W}_{1,p}[\mu](x)\le u(x)\le K_1{\bf W}_{1,p}[\mu](x),
\end{equation}
for all $x$ in $\mathbb{R}^N$, where the constant $K_1$ is as in Theorem \ref{TH4}.
Furthermore any $p$-superharmonic function $u$ in $\mathbb{R}^N$, such that $\inf_{\mathbb{R}^N}u=0$ satisfies \eqref{07063}  with $\mu=-\Delta_pu$. 
\end{theorem}
\begin{proof}[Proof of Theorem \ref{MT2}]Let $\{u_m\}_{m\in \BBN}$ be a 
sequence of $p$-superharmonic  solutions of the following problems  
\begin{equation*}
\begin{array}{ll}
  \phantom{--,,}- {\Delta _p}u_0 = \mu\qquad\text{in }~\mathcal{D}'(\mathbb{R}^N),  \\ \phantom{  - {\Delta _p}}
 \inf_{\mathbb{R}^N}u_0=0, 
 \end{array}
\end{equation*}
 and, for $m\in\BBN$,
\begin{equation*}
\begin{array}{ll}
    \phantom{--,,}- {\Delta _p}u_{m+1}  =P_{l,a,\beta}(u_{m})+ \mu~~\text{in }~\mathcal{D}'(\mathbb{R}^N),  \\ \phantom{  - {\Delta _p}}
 \inf_{\mathbb{R}^N}u_{m+1}=0. 
 \end{array} 
\end{equation*}
Clearly, we can assume that $\{u_m\}$ is nondecreasing. 
By Theorem \ref{07061} we have
\begin{equation*}\begin{array}{ll}
    u_0\leq K_1 {\bf W}_{1,p}[\mu],\\[2mm]
    \!\!\!\!\!\!\!u_{m+1}\leq K_1 {\bf W}_{1,p}[P_{l,a,\beta}(u_{m})+\mu]~~\forall m \in 
    \mathbb{N}.
  \end{array}  \end{equation*} Thus, by Theorem \ref{TH3-B} with $f\equiv 0$,
 there exists $M>0$ depending on $N,p,l,a,\beta,K_1$ and $R$ such that  $P_{l,a,\beta}(4c_pK_1{\bf W}_{1 ,p}[\omega])\in L^1(\mathbb{R}^N)$ and
 \begin{equation}\label{1335a}
    {u_m} \le 2c_pK_1 {\bf W}_{1 ,p}[\omega] \quad\forall m\in\BBN,
    \end{equation}
    provided that\[||M_{ p}^{\frac{{(p - 1)(\beta  - 1)}}{\beta }}[\mu ]|{|_{{L^\infty }({\mathbb{R}^N})}} \le M,\]
 where $\omega=M||{\bf M}_{ p}^{\frac{(p-1)(\beta-1)}{\beta}} [\chi_{B_{R}}] ||^{-1}_{{L^\infty }(\mathbb{R}^N)}\chi_{B_{R}}+\mu$.
This implies that $\{u_m\}$ is well defined and nondecreasing. Thus, $\{u_m\}$ converges a.e in $\mathbb{R}^N$ to some $p$-superharmonic  function $u$ which satisfies \eqref{MT2b}  in $\mathbb{R}^N$. Furthermore, we deduce  from \eqref{1335a} and the monotone convergence theorem that $P_{l,a,\beta}(u_{m})\to P_{l,a,\beta}(u)$ in $L^1(\mathbb{R}^N)$. Finally, by Theorem \ref{07062} we conclude that $u$ is a p-superharmonic solution of \eqref{MT2a}.\\
     Conversely, assume that \eqref{MT2a} admits a nonnegative renormalized solution $u$. By Theorem \ref{07061} there holds 
   $$u(x)\geq \frac{1}{K_1}{\bf W}_{1,p}[P_{l,a,\beta}(u)+\mu](x)~~\textrm{ for all } x\in \mathbb{R}^N.$$
   Hence, we obtain \eqref{MT2c} from Theorem \ref{TH1}, Part 1, (ii).
\end{proof}

 \section{Hessian equations}
In this section $\Omega\subset\mathbb{R}^N$ is either a bounded domain with a $C^2$ boundary or the whole $\mathbb{R}^N$. For $k=1,...,N$ and $u\in C^2(\Omega)$ the k-hessian operator $F_k$ is defined by 
 $$F_k[u]=S_k(\lambda(D^2u)),$$
 where $\lambda(D^2u)=\lambda=(\lambda_1,\lambda_2,...,\lambda_N)$ denotes the eigenvalues of the Hessian matrix of second partial derivative $D^2u$ and $S_k$ is the k-th elementary symmetric polynomial that is \[{S_k}(\lambda ) = \sum\limits_{1 \le {i_1} < ... < {i_k} \le N} {{\lambda _{{i_1}}}...{\lambda _{{i_k}}}}. \]
 We can see that \[{F_k}[u] = {\left[ {{D^2}u} \right]_k},\]
 where for a matrix $A=(a_{ij})$, $[A]_k$ denotes the sum of the k-th principal minors.
We assume that  $\partial \Omega$ is uniformly (k-1)-convex, that is $$S_{k-1}(\kappa ) \geq c_0 >0~on ~~ \partial\Omega,$$
     for some positive constant $c_0$, where $\kappa= (\kappa_1,\kappa_2,...,\kappa_{n-1})$ denote the principal curvatures of $\partial \Omega$ with respect to its inner normal.
   \begin{definition}\label{k-conv}
An upper-semicontinuous function $u:\Omega\to [-\infty,\infty)$ is k-convex (k-subharmonic) if, for every open set $\Omega'\subset\overline \Omega'\subset\Omega$ and for every function $v\in C^2(\Omega')\cap C(\overline{\Omega'})$ satisfying $F_k[v]\leq 0$ in $\Omega'$, the following implication is true 
           $$u\leq v ~on ~\partial\Omega' ~~~\Longrightarrow~~~ u\leq v~~in ~~\Omega'.$$
           We denote by $\Phi^k(\Omega)$ the class of all $k$-subharmonic functions in $\Omega$ which are not identically equal to $-\infty$.   
   \end{definition} 
 The following weak convergence result for $k$-Hessian operators proved in \cite{TW2} is fundamental in our study. 
 \begin{theorem}\label{TH6} Let $\Omega$ be either a bounded uniformly (k-1)-convex in $\mathbb{R}^N$ or the whole  $\mathbb{R}^N$. For each $u\in \Phi^k(\Omega)$, there exist a nonnegative Radon measure $\mu_k[u]$ in $\Omega$ such that \smallskip

 \noindent {\bf 1} $\mu_k[u]=F_k[u]$ for $u\in C^2(\Omega)$.\smallskip
 
 \noindent {\bf 2}  If $\{u_n\}$ is a sequence of k-convex functions which converges a.e to $u$, then $\mu_{k}[u_n]\rightharpoonup\mu_{k}[u] $ in the weak sense of measures.

 \end{theorem}
 
As in the case of quasilinear equations with measure data, precise estimates of solutions of k-Hessian equations with measures data are expressed in terms of Wolff potentials. The next results are proved in \cite{TW2,La,PhVe}. 
 \begin{theorem}\label{TH5}Let $\Omega\subset \mathbb{R}^N$ be a bounded $C^2$, uniformly (k-1)-convex  domain. Let $\varphi$ be a nonnegative continuous function on $\partial \Omega$ and  $\mu$ be a nonnegative Radon measure. Suppose that $\mu$ can be decomposed under the form 
 $$\mu=\mu_1+f$$
where $\mu_1$ is a measure with compact support in $\Omega$ and $f\in L^q(\Omega)$ for some $q>\frac{N}{2k}$ if $k\le\frac{N}{2}$, or $p=1$ if $k>\frac{N}{2}$. Then there exists a nonnegative function $u$ in $\Omega$ such that $-u\in \Phi^k(\Omega)$, continuous near $\partial \Omega$ and u is a solution of the problem 
 \[\begin{array}{lll}
  {F_k}[-u] = \mu \qquad&\;in\;\Omega,  \\ 
  \phantom{{F_k}[-]}u = \varphi \qquad&\;on\;\partial\Omega.
  \end{array} \]
 Furthermore,  any nonnegative function $u$ such that  $-u\in \Phi^k(\Omega)$ which is continuous near $\partial \Omega$ and is a solution of above equation, satisfies 
 \begin{equation}
\frac{1}{K_2}{\bf W}^{\frac{d(x,\partial \Omega)}{8}}_{\frac{2k}{k+1},k+1}[\mu]\le u(x)\le K_2 \left({\bf W}^{2diam\,\Omega}_{\frac{2k}{k+1},k+1}[\mu](x)+\max_{\partial \Omega}\varphi \right),
 \end{equation}
 where  $K_2$ is a positive constant independent of $x,u$ and $\Omega$. 
 \end{theorem}
 \begin{theorem}\label{07065}
 Let $\mu$ be  a measure in $\mathfrak{M}^+(\mathbb{R}^N)$ and $2k<N$. Suppose that ${\bf W}_{\frac{2k}{k+1},k+1}[\mu]<\infty$ a.e. Then there exists $u$, $-u\in \Phi^k(\mathbb{R}^N)$ such that $\inf_{\mathbb{R}^N}u=0$ and $F_k[-u]=\mu ~~\text{ in }~~ \mathbb{R}^N$
 and 
 \begin{equation}\label{07064}
 \frac{1}{K_2}{\bf W}_{\frac{2k}{k+1},k+1}[\mu](x)\le u(x)\le K_2{\bf W}_{\frac{2k}{k+1},k+1}[\mu](x),
 \end{equation}
 for all $x$ in $\mathbb{R}^N$, where the constant $K_2$ is the one of the previous Theorem.
Furthermore, if $u$ is a nonnegative function  such that $\inf_{\mathbb{R}^N}u=0$ and  $-u\in \Phi^k(\mathbb{R}^N)$, then \eqref{07064} holds  with $\mu=F_k[-u]$. 
 \end{theorem}

\noindent \begin{proof}[Proof of Theorem \ref{MT3}]
 We defined  a sequence of nonnegative functions $u_m$, continuous near $\partial \Omega$ and such that $-u_m\in \Phi^k(\Omega)$, by the following iterative scheme  
 \begin{equation}\label{k-hess-0}
\begin{array}{ll}
   F_k[-u_0] = \mu \qquad&\text{in }\;\Omega,  \\ 
   \phantom{   F_k[-]}
  u_0 = \varphi&\text{on }\;\partial\Omega,\\ 
  \end{array} 
 \end{equation}
 and,  for $m\geq 0$,
 \begin{equation}\label{k-hess-1}
  \begin{array}{ll}
   F_k[-u_{m+1}]  =P_{l,a,\beta}(u_{m})+ \mu\qquad&\text{in }\;\Omega,   \\ 
      \phantom{   F_k[-]}
  u_{m+1} = \varphi&\text{on }\;\partial\Omega.
  \end{array}
 \end{equation}
Clearly, we can assume that $\{u_m\}$ is nondecreasing, see \cite{PhVe2}. By Theorem \ref{TH5} we have 
 \begin{equation}\label{k-hess-2}  \begin{array}{ll}
    \phantom{-,} \chi_{\Omega}u_0\leq K_2 {\bf W}^{R}_{\frac{2k}{k+1},k+1}[\mu]+b_0,\\[2mm]
     \chi_{\Omega}u_{m+1}\leq K_2 {\bf W}^{R}_{\frac{2k}{k+1},k+1}[P_{l,a,\beta}(u_m)+
     \mu]+b_0,
   \end{array}   \end{equation}
 where $b_0=K_2\max_{\partial\Omega}\varphi$ and $R=2 \,diam\, (\Omega)$.\\
Then, by Theorem \ref{TH3} with $f=b_0$ and $\varepsilon=a$,
  there exists $M_1>0$ depending on $N,k,l,a,\beta, K_2$ and $R$ such that ${P_{l,a,\beta}\left({ 4K_2 {\bf W}_{\frac{2k}{k+1},k+1}^{R}[\omega_1 ] + 2g+b_0}\right) \in L^1(\Omega)}$ and
  \begin{equation}\label{1339}
     {u_m(x)} \le 2K_2 {\bf W}_{\frac{2k}{k+1},k+1}^R[\omega_1 ](x) + g+b_0\qquad\forall x\in \Omega, \,\forall m\geq 0,
     \end{equation}
     provided that
     \[||M_{ 2k,R}^{\frac{{k(\beta  - 1)}}{\beta }}[\mu ]|{|_{{L^\infty }({\mathbb{R}^N})}} \le M_1 ~\text{ and } ||M_{ 2k,R}^{\frac{{k(\beta  - 1)}}{\beta }}[P_{l,2a,\beta}(b_0) ]|{|_{{L^\infty }({\mathbb{R}^N})}} \le M_1,\]
where $\omega_1=M_1||{\bf M}_{2k}^{\frac{(p-1)(\beta-1)}{\beta}} [1] ||^{-1}_{{L^\infty }(\mathbb{R}^N)}+\mu$, $\omega_2=M_1||{\bf M}_{2k}^{\frac{(p-1)(\beta-1)}{\beta}} [1] ||^{-1}_{{L^\infty }(\mathbb{R}^N)}+P_{l,2a,\beta}(b_0)$ and $g=2K_2{\bf W}_{\frac{2k}{k+1},k+1}^{R}[\omega_2]$.\\ Since $\gw_2$ is constant, $g$ has the same property and actually $g=K_2(|B_1|\gw_2)^{\frac{1}{k}}R^2$. 
On the other hand, one can find constants $M_2$ depending on $N,k,l,a,\beta,R$ and $M_1$ such that if $\max_{\partial \Omega}\varphi\le M_2$, then $||M_{ 2k,R}^{\frac{{k(\beta  - 1)}}{\beta }}[P_{l,2a,\beta}(b_0) ]|{|_{{L^\infty }({\mathbb{R}^N})}} \le M_1$. \\
    Hence, we  deduce from \eqref{1339} that ${P_{l,a,\beta}\left({2 K_2 {\bf W}_{\frac{2k}{k+1},k+1}^{R}[\mu] + b}\right) \in L^1(\Omega)}$  and
    \begin{equation}\label{1339a}
         {u_m(x)} \le 2K_2 {\bf W}_{\frac{2k}{k+1},k+1}^R[\mu ](x) + b\qquad\forall x\in \Omega, \,\forall m\geq 0,
         \end{equation}
   for some constant $b$ ($=2g+b_0$)  depending on $N,k,l,a,\beta,R$ and $M_1$. 
Note that because we can write 
$$\omega= P_{l,a,\beta}(u_m)+\mu = \left(\mu_1 + \chi_{\Omega_\delta}P_{l,a,\beta}(u_m)\right)+\left((1-\chi_{\Omega_\delta})P_{l,a,\beta}(u_m)+f\right), $$  
where $\Omega_\delta=\{x\in\Omega:d(x,\partial \Omega)>\delta\}$ and  $\delta>0$ is small enough and  since $u_m$ is continuous near $\partial \Omega$, then  $\omega$ satisfies the assumptions of the data  in Theorem \ref{TH5}.   Therefore the sequence $\{u_m\}$ is well defined and nondecreasing. Thus, $\{u_m\}$ converges a.e in $\Omega$ to some function $u$ for which \eqref{MT3b} is satisfied in $\Omega$. Furthermore, we deduce  from \eqref{1339a} and the monotone convergence theorem that $P_{l,a,\beta}(u_m)\to P_{l,a,\beta}(u)$ in $L^1(\Omega)$. Finally, by Theorem \ref{TH6}, we obtain  that $u$ satisfies (\ref{MT3a}) and (\ref{MT3b}).\\
       Conversely, assume that \eqref{MT3a}  admits nonnegative solution $u$, continuous near $\partial \Omega$, such that $-u\in \Phi^k(\Omega)$ and $P_{l,a,\beta}(u)\in L^1(\Omega)$. Then by Theorem \ref{TH5} we have 
         $$u(x)\geq \frac{1}{K_2}{\bf W}^{\frac{d(x,\partial\Omega)}{8}}_{\frac{2k}{k+1},k+1}[P_{l,a,\beta}(u)+\mu](x)~~\textrm{ for all } x\in \Omega.$$
          Using the part 2 of Theorem \ref{TH1}, we conclude that \eqref{MT3c} holds.
 \end{proof}\medskip\\
 \begin{proof}[Proof of Theorem \ref{MT4}] 
We define  a sequence of nonnegative functions $u_m$ with $-u_m\in \Phi^k(\mathbb{R}^N)$, by the following iterative scheme  
  \begin{equation}\label{k-hess-3}
 \begin{array}{ll}
      \phantom{ }  F_k[-u_0] = \mu \qquad&\text{in }\mathbb{R}^N  \\ 
  \inf_{\mathbb{R}^N}u_0=0, 
   \end{array} 
  \end{equation}
  and,  for $m\geq 0$,
  \begin{equation}\label{k-hess-4}
   \begin{array}{ll}
    F_k[-u_{m+1}]  =P_{l,a,\beta}(u_{m})+ \mu\qquad&\text{in }\;\mathbb{R}^N   \\ 
   \inf_{\mathbb{R}^N}u_{m+1} =0.
   \end{array}
  \end{equation}
 Clearly, we can assume that $\{u_m\}$ is nondecreasing. By Theorem \ref{07065}, we have 
  \begin{equation}\label{k-hess-5}  \begin{array}{ll}
        \phantom{  F_k}     u_0\leq K_2 {\bf W}_{\frac{2k}{k+1},k+1}[\mu],\\[2mm]
     u_{m+1}\leq K_2 {\bf W}_{\frac{2k}{k+1},k+1}[P_{l,a,\beta}(u_m)+
      \mu].
    \end{array}   \end{equation}
  Thus, by Theorem \ref{TH3-B} with $f \equiv 0$, 
   there exists $M>0$ depending on $N,k,l,a,\beta$ and $R$ such that ${P_{l,a,\beta}\left({4 K_2 {\bf W}_{\frac{2k}{k+1},k+1}[\omega ]}\right) \in L^1(\mathbb{R}^N)}$,
   \begin{equation}
      {u_m} \le 2K_2 {\bf W}_{\frac{2k}{k+1},k+1}[\omega ] \qquad ~\forall m\geq 0,
      \end{equation}
      provided that
      $||M_{ 2k}^{\frac{{k(\beta  - 1)}}{\beta }}[\mu ]|{|_{{L^\infty }({\mathbb{R}^N})}} \le M ,$
      where   $\omega=M||{\bf M}_{2k}^{\frac{k(\beta-1)}{\beta}}[\chi_{B_R}]||_{L^\infty(\mathbb{R}^N)}^{-1} \chi_{B_R}+\mu$.\\
    Therefore the sequence $\{u_m\}$ is well defined and nondecreasing. By arguing as in the proof of theorem \ref{MT3}
    we obtain  that $u$ satisfies (\ref{MT4a}) and (\ref{MT4b}).\\       Conversely, assume that \eqref{MT4a}  admits a nonnegative solution $u$ and  $-u\in \Phi^k(\mathbb{R}^N)$ such that $P_{l,a,\beta}(u)\in L^1_{loc}(\mathbb{R}^N)$, then by Theorem \ref{07065} we have 
          $$u\geq \frac{1}{K_2}{\bf W}_{\frac{2k}{k+1},k+1}[P_{l,a,\beta}(u)+\mu].$$
           Using the part 1, (ii) of Theorem \ref{TH1}, we conclude that \eqref{MT4c} holds.
  \end{proof}
\end{document}